\documentclass[12pt,reqno]{amsart}

\usepackage{float} 

\usepackage{amsthm}
\usepackage{epsfig,amssymb,amsmath,version}
\usepackage{amssymb,version,graphicx,fancybox,mathrsfs}

\usepackage{url,hyperref}
\usepackage{subfigure}
\usepackage{color}
\usepackage{stmaryrd}
\usepackage{multirow}
\usepackage{booktabs,siunitx}
\usepackage{booktabs}
\usepackage{multicol,epstopdf}
\usepackage{xypic}
\usepackage[]{graphicx}
\usepackage[all]{xy}
\usepackage{tikz,ifthen}
\usetikzlibrary{positioning, math, decorations.markings, arrows.meta}
\usetikzlibrary{calc,shapes,cd}
\usetikzlibrary{knots} 
\usepgfmodule{decorations}
\allowdisplaybreaks
\allowdisplaybreaks[4]

\graphicspath{{./Figures/} } 

\begin{document}
\bibliographystyle{plain}

\title{Kauffman bracket intertwiners and the volume conjecture}

\author[Zhihao Wang]{Zhihao Wang}
\address{Zhihao Wang, School of Physical and Mathematical Sciences, Nanyang Technological University, 21 Nanyang Link Singapore 637371}
\email{ZHIHAO003@e.ntu.edu.sg}


 \maketitle


\newtheorem{thm}{Theorem}[section]    
\newtheorem{lmm}[thm]{Lemma}          
\newtheorem{conj}[thm]{Conjecture}  
\newtheorem{cor}[thm]{Corollary}    
\newtheorem{prop}[thm]{Proposition}   
\newtheorem{rem}[thm]{Remark}    
\newtheorem{exam}[thm]{Example}    
%
\theoremstyle{definition}
\newtheorem{defn}[thm]{Definition}    

\def \cb {\color{blue}}
\def \cred {\color{red}}
\def \cbf {\color{blue}\bf}
\def \credf {\color{red}\bf}
\definecolor{ligreen}{rgb}{0.0, 0.3, 0.0}
\def \cg {\color{ligreen}}
\def \cgf {\color{ligreen}\bf}
\definecolor{darkblue}{rgb}{0.0, 0.0, 0.55}
\def \dbf {\color{darkblue}\bf}
\definecolor{anti-flashwhite}{rgb}{0.55, 0.57, 0.68}
\def \afw {\color{anti-flashwhite}}


\def \ri {{\rm i}}
\newcommand{\bs}[1]{\boldsymbol{#1}}

\begin{abstract}The volume conjecture relates the quantum invariant and the hyperbolic geometry.
Bonahon-Wong-Yang introduced a new version of the volume conjecture by using the intertwiners  between two isomorphic irreducible representations of the skein algebra.
The intertwiners are built from surface diffeomorphisms; they  formulated the volume conjecture when diffeomorphisms are pseudo-Anosov. In this paper, we explicitly calculate all the intertwiners
for the closed torus using  
an algebraic embedding from the skein algebra of the closed torus to a quantum torus \cite{FG},  and show the  limit superior related to the trace of these intertwiners is zero. Moreover, we consider the periodic
diffeomorphisms for surfaces with  negative Euler characteristic, and  conjecture the corresponding limit  is zero because the simplicial volume of the mapping tori for periodic
diffeomorphisms is zero.  For the once punctured torus, we make precise calculations for intertwiners and prove our conjecture.
\end{abstract}


\maketitle

\section{Introduction}

In this paper, we first discuss  irreducible representations  for  skein algebras of the closed torus and the once punctured torus, which is related to Bonahon-Wong's work \cite{BW2,BW7,BW3,BW4}. They explored the connection between irreducible representations of skein algebras and the character variety related to the fundamental group of a surface. In section \ref{seeee}, we give more detailed discussions about this connection for  the closed torus and the once punctured torus. 

A profound result of the skein algebra is the unicity theorem, which was conjectured by Bonahon-Wong and was proved by Frohman-Kania-Bartoszynska-L{\^e} \cite{TLe2}. Based on this result there is an increased focus on the
Azumaya locus. Ganev-Jordan-Safronov  proved that the smooth part of the character variety lives in the
Azumaya locus when the surface is closed  \cite{IDP}. In section \ref{seeee}, we give  an explicit description for the
Azumaya locus for the skein algebra of the closed torus.

Let $S$ be an oriented surface, let $\varphi$ be a diffeomorphism for $S$,  and let $q_n = e^{2\pi i/n}$ with $(q_n)^{1/2} = e^{\pi i/n}$ and $n$ odd.  Using these data, Bonahon-Wong-Yang built  a sequence of intertwiners between irreducible representations of the skein algebra of $S$ \cite{BW5,BW6}.  When $S$ has negative
Euler characteristic
  and $\varphi$ is pseudo-Anosov, they formulated the volume conjecture using these intertwiners:
 \begin{equation*}
 \lim_{n\text{\;} odd \rightarrow \infty} \frac{1}{n} \log |\text{Trace} \Lambda^{q_n}_{\varphi,\gamma}| = \frac{1}{4\pi}vol_{hyp}(M_\varphi),
 \end{equation*}
   where $vol_{hyp}(M_\varphi)$ is the volume of the complete hyperbolic metric of the mapping torus $M_{\varphi}$.

\def \det {\text{det}}

We explicitly  compute the intertwiners corresponding to all diffeomorphisms of the closed  torus using  
an algebraic embedding from the skein algebra of the closed torus to a quantum torus \cite{FG}, see section \ref{seeee} for more details. The representation theory for this quantum torus is  well-studied.
 We prove almost all the irreducible representations of this quantum torus can be resticted to irreducible representations of  the skein algebra of the closed torus. So intertwiners  between two isomorphic irreducible representations of this quantum torus  are also the intertwiners between irreducible representations for the skein algebra of the closed torus. 
These intertwiners are built when the quantum parameter $q$ for the skein algebra is a primitive root of unity of odd order. We use $\Lambda_n$ to denote the intertwiner obtained as above when the quantum parameter is $q_n = e^{2\pi i/n}$ with $(q_n)^{1/2} = e^{\pi i/n}$ and $n$ odd. We also normalize  $\Lambda_n$ such that $|\text{det}(\Lambda_n)| = 1$. Then we prove the following Theorem, please refer to Theorem \ref{thm4.1} for a more detailed version.

\begin{thm}\label{ttttttt}
Let $\Lambda_n$ be defined as above, then we have 
\begin{equation*}
\limsup_{\text{odd }n\rightarrow \infty} \frac{\log(|{\rm{Trace}}\Lambda_n|)}{n} = 0.
\end{equation*}
\end{thm}

The
volume conjecture was first introduced by Kashaev \cite{Ka},  and  was rewritten and generalized to the non-hyperbolic case by Hitoshi Murakami and Jun Murakami \cite{MM}  using the simplicial volume.

Bonahon-Wong-Yang only formulated the conjecture when  the diffeomorphisms are  pseudo-Anosov for 
surfaces with negative Euler characteristic.   In this paper, we broaden the scope of the conjecture to include periodic diffeomorphisms.
When
$\varphi$ is a periodic diffeomorphism for the surface $S$, the corresponding mapping torus $M_{\varphi}$ is a Seifert manifold
whose simplicial volume is zero. So we conjecture the limits are zero for  periodic diffeomorphisms.
   We prove our conjecture for the once punctured torus, which serves as examples to confirm the limit is the simplicial volume of the corresponding mapping torus.

Let $S$ be an oriented surface with negative Euler characteristic, and let $\varphi$ be a  periodic diffeomorphism for $S$. According to page 371 in \cite{BM}, $\varphi$ fixes a point in the Teichm{\"u}ller space of $S$. This fixed point in the Teichm{\"u}ller space offers a smooth $\varphi$-invariant character $\gamma$
 (that is $\gamma$ is a group homomorphism from $\pi_1(S)$ to $SL(2,\mathbb{C})$ such that $\gamma\varphi_{*}$ and $\gamma$ have the same character, where $\varphi_{*}$ is the isomorphism from $\pi_1(S)$ to $\pi_1(S)$ induced by $\varphi$). 
Suppose the quantum parameter for the skein algebra is $q_n = e^{2\pi i/n}$ with $(q_n)^{1/2} = e^{\pi i/n}$ and $n$ odd. For each puncture $v$ of $S$, 
we choose a complex number $p_v$ such that $p_v = p_{\varphi(v)}$ and $T_n(p_v) = -\text{Trace}(\gamma(\alpha_v))$, where $T_n$ is the $n$-th Chebyshev polynomial of the first type and $\alpha_v$ is the element in $\pi_1(S)$ going around puncture $v$. According to Theorem \ref{thm1.1}, we know $\gamma$ and $p_v$ uniquely determine
an irreducible representation $\rho$ of the skein algebra. Let $\varphi_{\sharp}$ be the isomorphism from the skein algebra of $S$ to itself induced by $\varphi$. Since both $\gamma$ and $p_v$ are $\varphi$-invariant, we have 
$\rho$ and $\rho\varphi_{\sharp}$ are isomorphic according to Theorem \ref{uunicity}. Thus there exists the intertwiner
$\Lambda^{q_n}_{\varphi,\gamma}$ between these two isomorphic irreducible representations. We normalize it such that 
$|\text{det}(\Lambda^{q_n}_{\varphi,\gamma})| = 1$. Then we formulate the following conjecture, please refer to Conjecture \ref{conj4.2} for a more detailed version.

\begin{conj}\label{con1}
Let $S$ be a surface with negative Euler characteristic, let $\varphi$ be a  periodic diffeomorphism for $S$, and let
 $ \Lambda^{q_n}_{\varphi,\gamma}$ be defined as above, then we have
\begin{equation*}
 \lim_{n\text{\;} odd \rightarrow \infty} \frac{1}{n} \log |{\rm{Trace}} \Lambda^{q_n}_{\varphi,\gamma}| = 0.
 \end{equation*}
\end{conj}

In Theorem \ref{th4.1}, we prove the limit in Conjecture \ref{con1}  is less than or equal to zero if it exists by using the periodic property. It seems like we are half way there to prove our conjecture. But proving that the limit is greater than or equal to zero is  harder, which is actually related to an interesting question raised by Gerald Myerson \cite{GM} and Terry Tao \cite{TT}.  By direct calculations and using some  conclusions in \cite{TK,GM},
we prove the above conjecture for some special cases:

\begin{thm}\label{8888}
For any surface with negative Euler characteristic
if $\varphi$ is of order $2^m$ where $m$ is any positive integer, we have
$$ \lim_{n\text{\;} odd \rightarrow \infty} \frac{1}{n} \log |{\rm{Trace}} \Lambda^{q_n}_{\varphi,\gamma_{\varphi}}| = 0.$$
\end{thm}

\begin{thm}\label{6666}
Conjecture \ref{con1} holds if $S$ is  the once punctured torus.
\end{thm}

Plan of the paper: In section \ref{sec22}, we introduce the Kauffman bracket skein algebra, the classical shadow, the volume conjecture, and the Chekhov-Fork algebra, etc. Section \ref{seeee} is about the discussion on the irreducible representations of skein algebras of  the closed torus and  the once punctured torus. In section \ref{seccece}, we calculate the intertwiners for  the closed torus and prove  Theorem \ref{ttttttt}. In section \ref{sectiiii}, we  formulate our conjecture for  periodic diffeomorphisms and prove  Theorems \ref{8888} and \ref{6666}.

{\bf Acknowledgements}: The idea of considering periodic case was suggested by my supervisor Andrew James Kricker.
We would like to thank Andrew James Kricker, Jeffrey Weenink, Roland van der Veen and Xiaoming Yu  for  constructive discussion and help.
We wish to thank the referee most warmly for numerous suggestions that have improved the exposition
of this paper.
The research is supported by the NTU research scholarship from the Nanyang Technological University (Singapore).

\section{Preliminaries}\label{sec22}

\subsection{The SL(2,$\mathbb{C})$ character variety and the Kauffman bracket skein algebra}

Let $S$ be an oriented surface of finite type. 
 The corresponding character variety
$$\mathcal{X}_{SL(2,\mathbb{C})}(S)=\rm{Hom}(\pi_{1}(S), SL(2,\mathbb{C}))//SL(2,\mathbb{C})$$
 is the set of the group homomorphisms from the fundamental group of $S$ to $SL(2,\mathbb{C})$  with the equivalence relation that two homomorphisms are equivalent if and only if they have the same character   \cite{DB,MP,PS}.

The Kauffman bracket skein algebra $SK_{q^{1/2}}(S)$ of a surface $S$, as a vector space over the complex field $\mathbb{C}$, is generated by all isotopic framed links in $S\times [0, 1]$, subject to the skein relation:
$$K_1 = q^{-1/2} K_{\infty} + q^{1/2} K_{0},$$
where $K_1, K_{\infty}, K_{0}$ are  three links that differ in a small neighborhood as shown in Figure \ref{fg2.1}, and the trivial knot relation: $K\coprod \bigcirc = -(q + q^{-1}) K$ where $\bigcirc$ is a simple knot bounding a disk that has no intersection with $K$. For any two links $[L_{1}]$, $[L_{2}]$, the multiplication $[L_{1}][L_{2}]$ is defined  by stacking $L_{2}$ above $L_{1}$. Here $q^{1/2}$ is a nonzero complex number. 
The skein algebra $SK_{q^{1/2}}(S)$ is a quantization for the regular ring of the character variety $\mathcal{X}_{SL(2,\mathbb{C})}(S)$ \cite{DB}.

\begin{figure}[htbp]
	\centering
	\includegraphics[width=7cm]{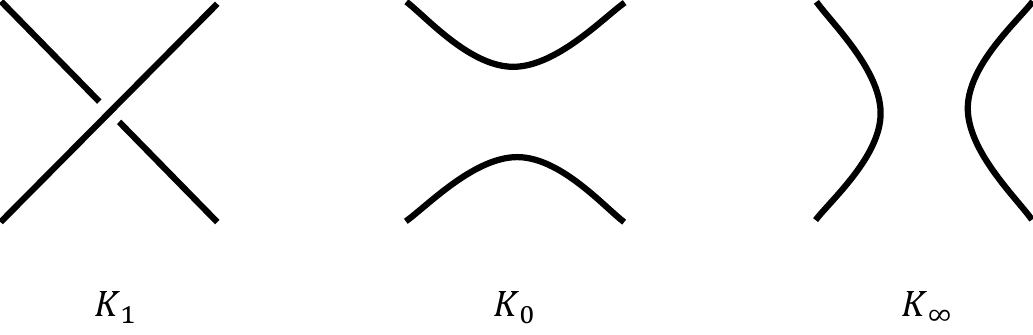}
	\caption{The Kauffman bracket skein relation.}\label{fg2.1}
\end{figure}

\subsection{Classical shadow and Unicity theorem}
We recall some notations and constructions for the classical  shadow \cite{BW2}. When $q$ is a primitive $n$-root of unity with $n$ odd and $(q^{1/2})^n = -1$,
  Bonahon and  Wong found a fascinating algebra homomorphism $T^{q^{1/2}}$
 from $SK_{-1}(S)$ to $SK_{q^{1/2}}(S)$, called the Chebyshev homomorphism.  Bonahon and  Wong proved that
 $\text{Im}(T^{q^{1/2}})$ is contained in the center of $SK_{q^{1/2}}(S)$. If $K$ is a simple knot with vertical framing,
 then $T^{q^{1/2}} ([K]) = T_n([K])$ where $T_n$ is the $n$-th Chebyshev polynomial of the first type.

 Let $\rho: SK_{q^{1/2}}(S)\rightarrow {\rm{End}}(V)$ be an irreducible representation of $SK_{q^{1/2}}(S)$.  Then there exists an algebra homomorphism
 $\kappa_{\rho}$ from $SK_{-1}(S)$ to $\mathbb{C}$ such that $\rho\circ T^{q^{1/2}}(X) = \kappa_{\rho}(X)Id_V$ for any $X$ in $SK_{-1}(S)$.
 According to \cite{DB}, there exists a unique character $[\gamma]\in \mathcal{X}_{SL(2,\mathbb{C})}(S)$
 such that $Tr^{\gamma} = \kappa_{\rho}$. Recall that $Tr^{\gamma}$ is an algebra homomorphism from
 $SK_{-1}(S)$ to $\mathbb{C}$ defined by $Tr^{\gamma}([K]) = -{\rm{Trace}}\gamma(K)$ where $[K]$
 is a simple knot. For every puncture $v$, we use $P_v$ to denote the loop going around this puncture.
 There is a complex number $p_v$ such that $\rho([P_v]) = p_v Id_V$. Then an irreducible representation
 of $SK_{q^{1/2}}(S)$ gives a character $[\gamma]$, called the classical shadow of this irreducible representation,  and puncture weights $\{p_v\}_{v}$,  with the relation that
 $-{\rm{Trace}}\gamma(\alpha_v) = T_n(p_v)$ where $\alpha_v$ denotes the element in the fundamental group of $S$  going around the puncture $v$.

 \begin{thm}[\cite{BW2,BW3,BW4,BW5}]\label{thm1.1}
 Let $q$ be a primitive $n$-root of unity with $n$ odd, and  $(q^{1/2})^n = -1$. Then an irreducible representation $\rho: SK_{q^{1/2}}\rightarrow {\rm{End}}(V)$ uniquely determines:

 (1) a  character $[\gamma]\in \mathcal{X}_{SL(2,\mathbb{C})}(S)$, represented by a group homomorphism $\gamma:\pi_1(S)\rightarrow SL(2,\mathbb{C})$;

 (2) a weight $p_v$ associated to each puncture $v$ of $S$ such that $T_n(p_v) = -{\rm{Trace}}\gamma(\alpha_v)$.

 Conversely, every data of a character $\gamma\in \mathcal{X}_{SL(2,\mathbb{C})}(S)$ and of puncture weights $p_v\in\mathbb{C}$ satisfying the above condition
 is realized by an irreducible representation $\rho: SK_{q^{1/2}}(S)\rightarrow {\rm{End}}(V)$.
 \end{thm}

\def \dim {\text{dim}}

It turns out that every character in an open dense subset of $\mathcal{X}_{SL(2,\mathbb{C})}(S)$ corresponds to a unique irreducible representation of the skein algebra.

 \begin{thm}[\cite{BW5,TLe2,IDP}]\label{uunicity}
 Suppose that $[\gamma]$ is in the smooth part of $\mathcal{X}_{SL(2,\mathbb{C})}(S)$ or, equivalently, that it is realized by an irreducible homomorphism
 $\gamma :\pi_1(S)\rightarrow SL(2,\mathbb{C})$. Then the irreducible representation $\rho: SK_{q^{1/2}}(S)\rightarrow {\rm{End}}(V)$ in Theorem \ref{thm1.1}
 is unique up to isomorphism of representations. This representation has dimension $\dim V=n^{3g+p-3}$ if $S$ has genus $g$ and $p$ punctures.
 \end{thm}

 \subsection{Volume conjecture for surface diffeomorphisms}
Bonahon-Wong-Yang constructed the so called Kauffman bracket intertwiners \cite{BW5,BW6}. They used these intertwiners to
 formulate the volume conjecture for surface diffeomorphisms. Here we recall their construction for Kauffman bracket intertwiners.

\def \Mod {\textup{Mod}}

 For a surface $S$, let $\varphi$ be a diffeomorphism of $S$.
Obviously $\varphi$ induces an isomorphism $\varphi_{*}$ from $\pi_{1}(S)$ to $\pi_{1}(S)$.
Then $\varphi_*$ induces an action on $\mathcal{X}_{SL(2,\mathbb{C})}(S)$ defined by
$\varphi^{*}([\gamma]) = [\gamma\varphi_{*}]$ where $\gamma$ is a representative for $[\gamma]$.
Although $\varphi_{*}$ is only defined up to conjugation, we have $\varphi^{*}$ is well-defined.
Actually the mapping class group $\Mod(S)$ acts on $\mathcal{X}_{SL(2,\mathbb{C})}(S)$.
We say an element $[\gamma]\in \mathcal{X}_{SL(2,\mathbb{C})}(S)$ is invariant under
a diffeomorphism $\varphi$, or the element it represents in $\Mod(S)$, if $\varphi^{*}([\gamma]) = [\gamma]$.

The algebra isomorphism induced by $\varphi$ from $SK_{q^{1/2}}(S)$ to itself is defined by
$\varphi_{\sharp}([K]) = [\varphi\times Id_{[0,1]} (K)]$ where $K$ is a framed link in $S\times [0,1]$. Actually
the mapping class group $\Mod(S)$ acts on $SK_{q^{1/2}}(S)$.

Let $\varphi$ be any  diffeomorphism for surface $S$, and let $[\gamma]\in \mathcal{X}_{SL(2,\mathbb{C})}(S)$ be
a $\varphi$-invariant  smooth character. For each puncture $v$, select a complex number $\theta_v$ such that ${\rm{Trace}}\gamma(\alpha_v) = -e^{\theta_v} - e^{-\theta_v}$. Since
$[\gamma]$ is $\varphi$-invariant, we can choose $\theta_v$ to be $\varphi$-invariant, that is, $\theta_v = \theta_{\varphi(v)}$. Then set
$p_v = e^{\frac{\theta_v}{n}} + e^{-\frac{\theta_v}{n}}$, we have $T_n(p_v) = -{\rm{Trace}}(\alpha_v)$ and  $\{p_v\}_{v}$ are invariant under the action of $\varphi$.
Suppose $\rho$ is an irreducible representation associated to $[\gamma]$ and puncture weights $p_v$.
Then $\rho\circ \varphi_{\sharp}$ is also an irreducible representation associated to $[\gamma]$ and puncture weights $p_v$. By the unicity theorem, we know
there exists an intertwiner $\Lambda^{q}_{\varphi,\gamma}$ such that
$$\rho\circ \varphi_{\sharp} (X) = \Lambda^{q}_{\varphi,\gamma} \circ \rho(X) \circ (\Lambda^{q}_{\varphi,\gamma})^{-1}$$
for every $X\in SK_{q^{1/2}}(S)$. We normalize the intertwiner such that $|\det(\Lambda^{q}_{\varphi,\gamma})| = 1$.

 \begin{conj}[\cite{BW5,BW6}]\label{conj4.1}
Let the pseudo-Anosov surface diffeomorphism $\varphi : S\rightarrow S$, the $\varphi$-invariant  smooth character $[\gamma]\in \mathcal{X}_{SL(2,\mathbb{C})}(S)$ and the $\varphi$-invariant puncture weights $p_v$ as above be given. For every odd $n$, consider the primitive $n$-root of unity $q_n = e^{2\pi i/n}$ and choose
$(q_n)^{1/2} = e^{\pi i/n}$.
 Then
 \begin{equation*}
 \lim_{n\text{\;} odd \rightarrow \infty} \frac{1}{n} \log |{\rm{Trace}} \Lambda^{q_n}_{\varphi,\gamma}| = \frac{1}{4\pi}vol_{hyp}(M_\varphi),
 \end{equation*}
   where $vol_{hyp}(M_\varphi)$ is the volume of the complete hyperbolic metric of the mapping torus $M_{\varphi}$.
\end{conj}


\subsection{Ideal triangulation and intertwiners obtained from Chekhov-Fock algebras}\label{newsec}

 Let $S$ be an oriented surface with punctures, and let $\tau = \{e_1,\cdots,e_m\}$ be an ideal triangulation for $S$, where $e_1,\cdots,e_m$ are non-isotopic disjoint embedded  arcs in $S$ connecting punctures such that all these arcs cut $S$ into triangles. We call $e_1,\cdots,e_m$ the edges of $\tau$.
An edge weight system for $\tau$ is an $m$-tuple, $a= (a_1,\cdots,a_m)$, where $a_i$ is a nonzero complex number for each $1\leq i\leq m$. The pair $(\tau,a)$ determines a character $[\overline{\gamma}]$ in $\mathcal{X}_{PSL(2,\mathbb{C})}(S)$, please refer to section 8 in \cite{BX} or section 3 in \cite{BW5} for more details.

For each ideal triangulation $\tau$, there is a Chekhov-Fock algebra $\mathcal{T}^{q}_{\tau}$ corresponding to $\tau$, where $q$ is a nonzero complex number. As an algebra over $\mathbb{C}$, the Chekhov-Fock algebra $\mathcal{T}^{q}_{\tau}$ is generated by $X_1^{\pm1},X_2^{\pm1},\dots,X_m^{\pm1}$ subject to the relations: $$X_i X_i^{-1} = X_i^{-1}X_i = 1,
X_i X_j = q^{2\sigma_{ij}} X_j X_i.$$
 Each $X_i$ is associated to the $i$-th edge in the ideal triangulation $\tau$, and $\sigma_{ij}$ is an integer determined by  $\tau$, see \cite{BX,BW1,Liu} for more details.
If we replace $q$ with $q^{\frac{1}{4}}$, we get the so called  Chekhov-Fock square root algebra $\mathcal{T}^{q^{\frac{1}{4}}}_{\tau}$.
 It is well known that $\mathcal{T}^{q}_{\tau}$ is an Ore domain. We will use $\hat{\mathcal{T}}^{q}_{\tau}$ to denote the ring of fractions of $\mathcal{T}^{q}_{\tau}$ (that is the localization over all nonzero elements).

Let $\tau,\tau^{'}$ be any two ideal triangulations for $S$. Then there is an algebra isomorphism
 $\Phi^{q}_{\tau\tau^{'}}:\hat{\mathcal{T}}^{q}_{\tau^{'}}\rightarrow \hat{\mathcal{T}}^{q}_{\tau}$, called be the Chekhov-Fock coordinate
  change isomorphism \cite{Liu}.

For an ideal triangulation $\tau$,
 there are two operations. 
(1) Reindexing: obtain a new ideal triangulation $\tau^{'}$ by reindexing all the edges in $\tau$. (2) Diagonal exchange: for any $1\leq i\leq m$, define a new ideal triangulation 
$\tau^{'} = \{e_1^{'},\cdots,e_m^{'}\}$, where $e_j^{'} = e_j$ for every $j\neq i$ and 
$e_i^{'}$ is the other diagonal of the square formed by the two faces of $\tau$ that are adjacent  to $e_i$.

Let $\tau$ be an ideal triangulation, let $a= (a_1,\cdots,a_m)$ be an edge weight system for $\tau$.
Suppose  $\tau^{'} = \{e_1^{'},\cdots,e_m^{'}\}$ is obtained from $\tau$ by reindexing such that 
$e_i^{'} = e_{\sigma(i)}$ for $1\leq i\leq m$, where $\sigma$ is a permutation for $\{1,\cdots,m\}$.
Then we  define an edge weight system $a^{'}= (a_1^{'},\cdots,a_m^{'})$ for $\tau^{'}$ by setting
$a_i^{'} = a_{\sigma(i)}$ for $1\leq i\leq m$. If $\tau^{'}$ is obtained from $\tau$ by the
 diagonal exchange, we define an edge weight system $a^{'}$ for $\tau^{'}$ using formulas in Proposition 3 in \cite{Liu}.  We will say $a^{'}$ is an edge weight system for $\tau^{'}$ derived from the pair $(\tau,a)$. Then $(\tau^{'},a^{'})$ determines the same character in $\mathcal{X}_{PSL(2,\mathbb{C})}(S)$ as $(\tau,a)$ \cite{BX,BW5}.

A sequence of ideal triangulations $\tau^{(0)}, \tau^{(1)},\dots, \tau^{(k)}$ is called an ideal triangulation sweep if, for each $1\leq i\leq k-1$, we have $\tau^{(i+1)}$ is obtained from $\tau^{(i)}$ by  reindexing or the diagonal exchange.
A sequence of edge weight systems $a^{(0)},a^{(1)},\cdots,a^{(k)}$ is called an edge weight system sweep for the ideal triangulation sweep $\tau^{(0)}, \tau^{(1)},\dots, \tau^{(k)}$, if the edge weight system $a^{(i+1)}$ for $\tau^{(i+1)}$  is derived from 
$(\tau^{(i)},a^{(i)})$ for each $0\leq i\leq k-1$. Note that the sequence $a^{(0)},a^{(1)},\cdots,a^{(k)}$ is completely determined by $a^{(0)}$. If in addition $a^{(0)} = a^{(k)}$, we call the sequence $a^{(0)},a^{(1)},\cdots,a^{(k)}$ a periodic edge weight system for  the ideal triangulation sweep $\tau^{(0)}, \tau^{(1)},\dots, \tau^{(k)}$.

Suppose $q$ is a primitive $n$-root of unity with $n$ odd.
  Let $\varphi$ be an orientation
 preserving diffeomorphism for surface $S$, and 
  let $\tau = \tau^{(0)}, \tau^{(1)},\dots, \tau^{(k)} = \varphi(\tau)$ be an ideal triangulation sweep.
Suppose
  $a = a^{(0)},a^{(1)},\dots, a^{(k)} = a$ is a periodic edge weight system for $\tau^{(0)}, \tau^{(1)},\dots, \tau^{(k)}$ (the existence of the periodic edge weight system is guaranteed by Lemma 11 in \cite{BW5}), which defines a $\varphi$-invariant character
  $[\overline{\gamma}]\in \mathcal{X}_{PSL(2,\mathbb{C})}(S)$. Then,
for each puncture $v$, we can choose a nonzero complex number $h_v$ such that
$h_v = h_{\varphi(v)}$ for every puncture $v$ and $(h_v)^n = a_{i_1}a_{i_2}\dots a_{i_{j}}$ for every puncture  $v$  adjacent to the edges $e_{i_1}, e_{i_2},\dots, e_{i_{j}}$.
From Proposition 13 in \cite{BW5}, we know  $a$ and puncture weights $h_v$ uniquely determine an irreducible representation 
$\overline{\rho}:\mathcal{T}^{q}_{\tau}\rightarrow \text{End}(V)$ for the  Chekhov-Fock algebra $\mathcal{T}^{q}_{\tau}$
such that $\overline{\rho}(X_i^{n}) = a_i$ for $1\leq i\leq m$ and $\overline{\rho}(H_v)
=h_v$ for each puncture $v$, where $H_v$ is a central element in $\mathcal{T}^{q}_{\tau}$ associated to each puncture $v$.
  Let $\Phi^{q}_{\tau\varphi(\tau)}:\hat{\mathcal{T}}^{q}_{\varphi(\tau)}\rightarrow \hat{\mathcal{T}}^{q}_{\tau}$ be the Chekhov-Fock coordinate
  change isomorphism, and let $\Psi_{\varphi(\tau)\tau}^{q}: \hat{\mathcal{T}}^{q}_{\tau} \rightarrow \hat{\mathcal{T}}^{q}_{\varphi(\tau)}$
  be the algebra isomorphism induced by $\varphi$. Then $\overline{\rho}\simeq \overline{\rho}\circ \Phi^{q}_{\tau\varphi(\tau)}\circ \Psi_{\varphi(\tau)\tau}^{q}$, so
  there exists an intertwiner $\overline{\Lambda}^{q}_{\varphi,\overline{\gamma}}$
  such that
  $$\overline{\rho}\circ \Phi^{q}_{\tau\varphi(\tau)}\circ \Psi_{\varphi(\tau)\tau}^{q} (X) =
  \overline{\Lambda}^{q}_{\varphi,\overline{\gamma}} \circ \overline{\rho}(X) \circ (\overline{\Lambda}^{q}_{\varphi,\overline{\gamma}})^{-1}$$
  for every $X\in \mathcal{T}^{q}_{\tau}$.

Under certain  conditions, the trace of intertwiners obtained from Chekhov-Fock algebras  equals  
the trace of intertwiners obtained from skein algebras, see Theorem 16 in \cite{BW5}. We will use this equality to calculate the trace of intertwiners obtained from skein algebras for the once punctured torus in section \ref{sectiiii}.

 From now on, we will always assume $q^{1/2}$ is
a primitive $n$-root of $-1$ with $n$ odd.

\section{Irreducible representation construction for $SK_{q^{1/2}}(T^2)$ and $SK_{q^{1/2}}(S_{1,1})$}\label{seeee}

In order to get  Kauffman bracket intertwiners, we want to find the explicit irreducible representations associated to given characters and puncture weights.
Here we construct irreducible representations for skein algebras of the closed torus $T^2$ and the once punctured torus $S_{1,1}$.
In section \ref{seccece}, we will use these irreducible representations to calculate intertwiners for the closed torus.


\subsection{An algebraic embedding for $SK_{q^{1/2}}(T^2)$}
Let $\mathbb{C}[X^{\pm 1},Y^{\pm 1}]_{q^{1/2}}$  be the algebra generated by $X, X^{-1}, Y, Y^{-1}$, subject to the relations $XY=q YX, XX^{-1} = X^{-1}X = 1, YY^{-1}  = Y^{-1}Y = 1$.  Frohman-Gelca built an algebraic embedding \cite{FG}:
\begin{equation*}
\begin{split}
G_{q^{1/2}} : SK_{q^{1/2}}(T^{2})&\rightarrow \mathbb{C}[X^{\pm 1},Y^{\pm 1}]_{q^{1/2}}\\
(a,b)_{T}&\mapsto \theta_{(a,b)} + \theta_{(-a, -b)},
\end{split}
\end{equation*}
where $(a,b)_{T}$ is the simple link associated to two integers $a, b$, and $\theta_{(a,b)}= q^{-ab/2}X^{a}Y^{b}$. If $\rm{gcd}(a, b) = 1$  (with the convention that  gcd$(\pm1,0) =\text{gcd}(0, \pm1) = 1$), then $(a,b)_{T}$ is represented by the simple knot $(a, b)$ in $\mathbb{R}^{2}/\mathbb{Z}^{2}$ with vertical framing. If $\text{gcd}(a,b) = k$ and  $a = a^{'}k, b = b^{'}k$, then $(a,b)_{T} = T_{k}((a^{'},b^{'}))$ where $T_{k}$ is the $k$-th  Chebyshev polynomial  of the first  type.
We have
$\theta_{(a,b)}\theta_{(c,d)} = q^{1/2\begin{bmatrix}
a & b\\
c & d
\end{bmatrix}} \theta_{(a+c, b+d)}$, and $(\theta_{(a,b)})^{-1} = \theta_{(-a, -b)}$. Since $\theta_{(a,b)} + \theta_{(-a, -b)} = q^{-ab/2}(X^{a}Y^b + X^{-a}Y^{-b})$,  then
${\rm{Im}} G_{q^{1/2}} = \text{span}<X^{a}Y^b + X^{-a}Y^{-b}| (a, b) \in \mathbb{Z}\times \mathbb{Z}>$.

Let $T_{q^{1/2}}$ be the Chebyshev homomorphism from the skein algebra $SK_{-1}(T^2)$ to $SK_{q^{1/2}}(T^{2})$ defined in \cite{BW2}, and let $F_{q^{1/2}}: \mathbb{C}[X^{\pm 1},Y^{\pm 1}]_{-1}\rightarrow \mathbb{C}[X^{\pm 1},Y^{\pm 1}]_{q^{1/2}}$ defined by $X\mapsto X^{n}, Y\mapsto Y^{n}$. 
It is easy to check that we have $F_{q^{1/2}}G_{-1} = G_{q^{1/2}}T_{q^{1/2}}$.

\subsection{Irreducible representations for $SK_{q^{1/2}}(T^2)$}
Bonahon-Liu  described the irreducible representations of $\mathbb{C}[X^{\pm 1},Y^{\pm 1}]_{q^{1/2}}$ \cite{BX}. Let $V$ denote the $n$-dimensional vector space over the complex field with basis $e_{0}, e_{1},\dots, e_{n-1}$, and let $u,v$ be any two nonzero complex numbers. Set $\rho_{u,v}(X) e_i = uq^{i} e_i,\;
\rho_{u,v}(Y) e_i = v e_{i+1}$, where the indices are considered modulo $n$,
then $\rho_{u,v}$ is an irreducible representation.
Any irreducible representation of $\mathbb{C}[X^{\pm},Y^{\pm}]_{q^{1/2}}$ is isomorphic to  a representation  $\rho_{u,v}$, and $\rho_{u,v}\simeq \rho_{u^{'},v^{'}}$ if and only if $u^{n}=(u^{'})^{n}, v^{n}=(v^{'})^{n}$.

It is well-known that $\pi_{1}(T^2) = \mathbb{Z}\oplus\mathbb{Z} = \mathbb{Z}\alpha\oplus \mathbb{Z}\beta$ where $\alpha= (1, 0)$ and $\beta = (0, 1)$.
For any $[\gamma] \in \mathcal{X}_{SL(2,\mathbb{C})}(T^{2})$, we have $[\gamma]$ has a representative $\gamma$ such that
$\gamma(\alpha)$=
$\begin{pmatrix}
  \lambda_{1} & 0\\
  0 & \lambda_{1}^{-1}
\end{pmatrix}$ and
$\gamma(\beta)$=
$\begin{pmatrix}
  \lambda_{2} & 0\\
  0 & \lambda_{2}^{-1}
\end{pmatrix}$ because $\pi_{1}(T^2)$ is commutative.

\def \mod {\textup{mod}}

For any given character $[\gamma] \in \mathcal{X}_{SL(2,\mathbb{C})}(T^{2})$, the following Theorem offers a representation of $SK_{q^{1/2}}(T^2)$ whose classical shadow is $[\gamma]$.
For this Theorem, we use the fact that $ab + a + b\equiv \text{gcd}(a, b)$ $(\mod\text{ } 2)$ for any two integers $a,b$ (recall that $\text{gcd}(\pm1, 0) = \text{gcd}(0, \pm1) = 1$).

\begin{thm}\label{thm2.1}
Choose $u, v \in\mathbb{C}$ such that $u^n = -\lambda_1, v^n = -\lambda_2$ or $u^n = -\lambda_1^{-1}, v^n = -\lambda_2^{-1}$, then the classical shadow of $\rho_{u,v}G_{q^{1/2}}$ is $[\gamma]$.
\end{thm}
\begin{proof}
To show  the classical shadow of $\rho_{u,v}G_{q^{1/2}}$ is $[\gamma]$, it suffices to show
$\rho_{u,v}G_{q^{1/2}}(T_{q^{1/2}}((a, b)_{T})) \\= Tr^{\gamma}((a, b)_{T})Id_{V}$ for all $(a, b)_{T}\in SK_{-1}(T^{2})$. First we have
\begin{align*}
&\rho_{u,v}G_{q^{1/2}}(T_{q^{1/2}}((a, b)_{T})) = \rho_{u,v}(F_{q^{1/2}}G_{-1}((a,b)_{T})) \\
= &\rho_{u,v}(F_{q^{1/2}}(\theta_{(a,b)}+ \theta_{(-a, -b)}))=\rho_{u,v}(\theta_{(na, nb)}+\theta_{(-na, -nb)}) \\
= &\rho_{u,v}((-1)^{ab}X^{na}Y^{nb}+(-1)^{ab}X^{-na}Y^{-nb})\\
 = &(-1)^{ab}[(\rho_{u,v}(X))^{na}(\rho_{u,v}(Y))^{nb}+(\rho_{u,v}(X))^{-na}(\rho_{u,v}(Y))^{-nb}]\\
= &(-1)^{ab}[(u^n)^{a}(v^n)^{b}+(u^n)^{-a}(v^n)^{-b}]Id_V\\
= &(-1)^{ab+a+b}[\lambda_1^{a}\lambda_2^{b}+\lambda_1^{-a}\lambda_2^{-b}]Id_V.
\end{align*}
Suppose $\text{gcd}(a,b) = d$ and $a = a^{'}d, b = b^{'}d$, then we have
\begin{align*}
&Tr^{\gamma}((a, b)_{T}) = Tr^{\gamma}(T_d((a^{'}, b^{'}))) = T_d(Tr^{\gamma}((a^{'}, b^{'}))) \\
=&T_d(-\text{Trace}(\gamma((a^{'}, b^{'}))))=T_d(-\text{Trace}(\gamma(a^{'}\alpha+b^{'}\beta)))\\
=&T_d(-\text{Trace}((\gamma(\alpha))^{a^{'}}(\gamma(\beta))^{b^{'}}))
 =  T_d((-\lambda_{1}^{a^{'}}\lambda_{2}^{b^{'}}) + (- \lambda_{1}^{-a^{'}}\lambda_{2}^{-b^{'}}))\\
  = & (-\lambda_{1}^{a^{'}}\lambda_{2}^{b^{'}})^{d} + (- \lambda_{1}^{-a^{'}}\lambda_{2}^{-b^{'}})^{d}
   =  (-1)^{d}[\lambda_{1}^{da^{'}}\lambda_{2}^{db^{'}} +  \lambda_{1}^{-da^{'}}\lambda_{2}^{-db^{'}}]\\
   = & (-1)^{ab+a+b}[\lambda_1^{a}\lambda_2^{b}+\lambda_1^{-a}\lambda_2^{-b}].
\end{align*}

\end{proof}

We can easily get the following Theorem by using the representation theory.

\begin{thm}\label{thm2.2}
Under the same assumption as in Theorem \ref{thm2.1}, we have the following conclusions:

(a) if $\lambda_{1}\neq \pm1$ or $\lambda_{2}\neq \pm1$, the representation $\rho_{u,v} G_{q^{1/2}}$ is  irreducible.

(b) if $\lambda_{1} = \pm1$ and $\lambda_{2} = \pm1$, we have $V$ has only two irreducible subrepresentations $V_{1}, V_{2}$, and
$V=V_{1}\oplus V_{2}$, and $\dim(V_1)=(n+1)/2, \dim(V_2) = (n-1)/2$, especially $V_1=$
span$<e_0, e_1+e_{n-1}, e_{2} + e_{n-2},\dots, e_{(n-1)/2} + e_{(n+1)/2}>$ and $V_2=$
span$<e_1-e_{n-1}, e_{2} - e_{n-2},\dots, e_{(n-1)/2} - e_{(n+1)/2}>$ if $u=\pm1, v=\pm1$.
\end{thm}

%

\begin{rem}

The Azumaya locus of $SK_{q^{1/2}}(T^2)$ is a subset of $\mathcal{X}_{SL(2,\mathbb{C})}(T^2)$. An element in $\mathcal{X}_{SL(2,\mathbb{C})}(T^2)$ lives in the Azumaya locus  if it   corresponds to a unique irreducible repesentation of $SK_{q^{1/2}}(T^2)$ (the correspondence is the one in Theorem \ref{thm1.1}).
We know the PI-dimension of $SK_{q^{1/2}}(T^2)$ is $n$. Then
a character $[\gamma]\in \mathcal{X}_{SL(2,\mathbb{C})}(T^2)$ lives in the Azumaya locus if and only if
there exists an irreducible representation of $SK_{q^{1/2}}(T^2)$ of dimension $n$ whose classical shadow is $[\gamma]$.
So from Theorem \ref{thm2.2}, we get $[\gamma]$ lives in the Azumaya locus if and only if $\lambda_{1}\neq \pm1$ or $\lambda_{2}\neq \pm1$,
where $\gamma(\alpha)$=
$\begin{pmatrix}
  \lambda_{1} & 0\\
  0 & \lambda_{1}^{-1}
\end{pmatrix}$ and
$\gamma(\beta)$=
$\begin{pmatrix}
  \lambda_{2} & 0\\
  0 & \lambda_{2}^{-1}
\end{pmatrix}$.
A contemporaneous paper by Karuo-Korinman \cite{KK} considered instead the case when $q^{1/2}$ is an odd order root of unity; both cases were studied through similar methods. They proved the character lives in the Azumaya locus of the skein algebra of a closed surface if and only if  the character is noncentral.

In \cite{BW7},  Bonahon-Wong proved  the Witten-Reshetikhin-Turaev representation of the Kauffman bracket skein algebra is irreducible and whose
classical shadow is the trivial character. For the closed torus $T^2$, we use $V_{T^2}$ to denote the Witten-Reshetikhin-Turaev representation of $SK_{q^{1/2}}(T^2)$.
We know $\dim V_{T^2} = \frac{n-1}{2}$ with basis $v_1, v_2,\dots, v_{\frac{n-1}{2}}$ where $v_k$ is the skein in  the solid torus represented by $2k-2$ nontrivial parallel
closed curves, which are parallel to the core of the solid torus,  with the $(2k-2)$-th Jones-Wenzel idempotent inserted. In Theorem \ref{thm2.2} with $\lambda_1 = \lambda_2 = 1$, we have $V_2$ is isomorphic to $V_{T^2}$
as  representations for $SK_{q^{1/2}}(T^2)$, and the isomorphism is given by
$$e_{2k-1} - e_{n-2k+1}\mapsto v_k,\forall 1\leq k\leq\frac{n-1}{2}.$$

\end{rem}

\subsection{Irreducible representations for $SK_{q^{1/2}}(S_{1,1})$}


We  want to find the explicit irreducible representations of $SK_{q^{1/2}}(S_{1,1})$  corresponding to given characters and puncture weights.
 Let $\mathbb{C}_{q}[X_1^{\pm1}, X_2^{\pm1}, X_3^{\pm1}]$ be the algebra generated by $X_1, X_2, X_3$ subject to the relations:
$$X_1X_2 = qX_2X_1, X_2X_3 = q X_3X_2, X_3X_1 = qX_1X_2, X_iX_i^{-1} = X_i^{-1}X_i = 1.$$
We have $\mathbb{C}_{q}[X_1^{\pm1}, X_2^{\pm1}, X_3^{\pm1}] = \mathcal{T}^{q^{1/4}}_{\tau}(S_{1,1})$ where $\tau$ is the ideal triangulation in  Figure \ref{fgggg3}.
\begin{figure}[htbp]
  \centering
  \includegraphics[scale=0.5]{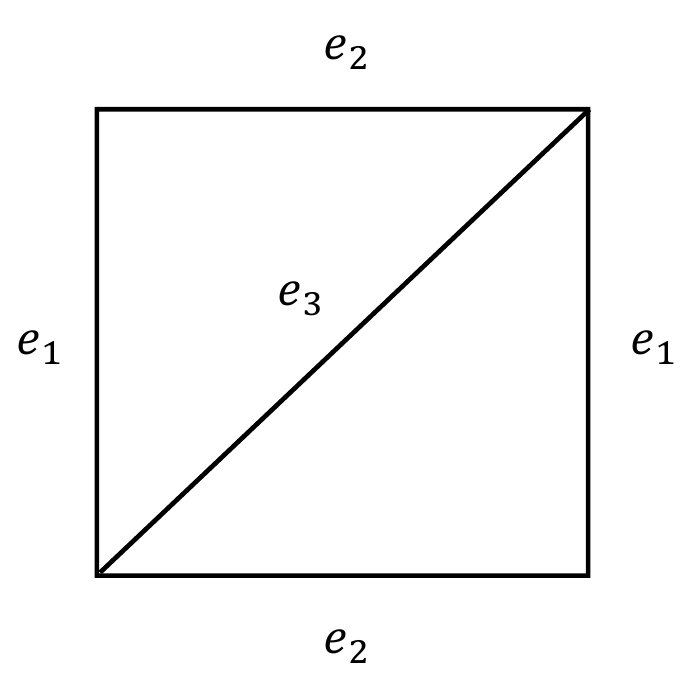}
\caption{}\label{fgggg3}
\end{figure}

We define the skeins $K_1, K_2, K_3$ in the skein algebra
$SK_{q^{1/2}}(S_{1,1})$ using Figure \ref{fg3.2}.
\begin{figure}[htbp]
  \centering
  \includegraphics[scale=0.6]{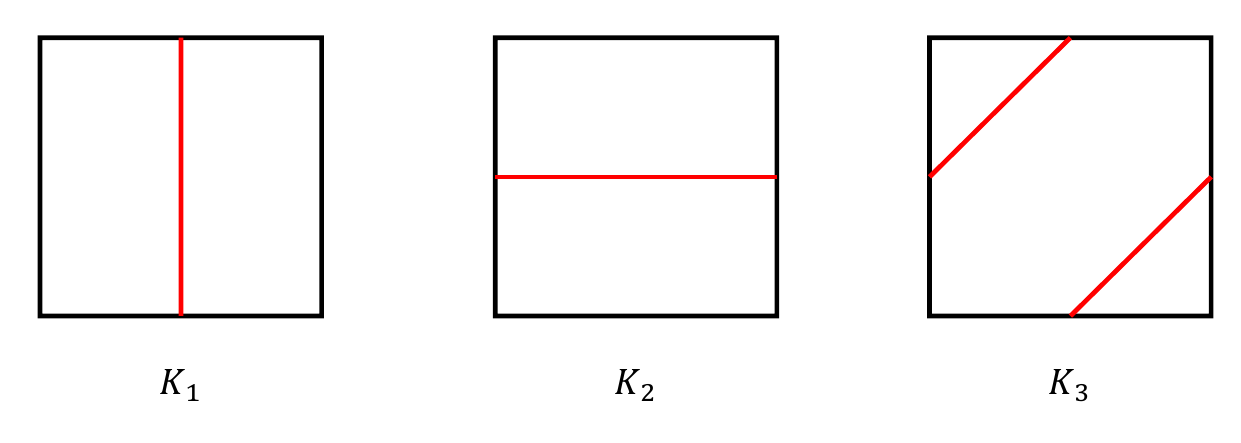}
\caption{}\label{fg3.2}
\end{figure}

According to \cite{BJ}, the algebra $SK_{q^{1/2}}(S_{1,1})$ is generated by $K_1,K_2,K_3$ subject to relations:
\begin{equation*}
\begin{split}
q^{-1/2}K_1 K_2 - q^{1/2}K_2 K_1 = (q^{-1} - q) K_3, \\
q^{-1/2}K_2 K_3 - q^{1/2}K_3 K_2 = (q^{-1} - q) K_1, \\
q^{-1/2}K_3 K_1 - q^{1/2}K_1 K_3 = (q^{-1} - q) K_2.
\end{split}
\end{equation*}
Let $P$ be the loop around the puncture in $S_{1,1}$, then
$$P = q^{-1/2}K_1 K_2 K_3 - q^{-1}K_1^{2} - qK_2^{2} - q^{-1}K_3^{2} + q + q^{-1}.$$

\begin{lmm}
There is an algebraic embedding $F: S_{q^{1/2}}(S_{1,1})\rightarrow \mathbb{C}_{q}[X_1^{\pm1}, X_2^{\pm1}, X_3^{\pm1}]$
such that
\begin{equation}
\begin{split}
F(K_1) &= [X_2 X_3] + [X_2^{-1}X_3^{-1}] + [X_2 X_3^{-1}],\\
F(K_2) &= [X_3 X_1] + [X_3^{-1}X_1^{-1}] + [X_3 X_1^{-1}],\\
F(K_3) &= [X_1 X_2] + [X_1^{-1}X_2^{-1}] + [X_1 X_2^{-1}],\\
F(P) &= [X_1^{2}X_2^{2}X_3^{2}] + [X_1^{-2}X_2^{-2}X_3^{-2}].
\end{split}
\end{equation}
\end{lmm}
\begin{proof}
Actually $F$ is just  the quantum trace map constructed in Theorem 11 in \cite{BW1} if we regard $\mathbb{C}_{q}[X_1^{\pm1}, X_2^{\pm1}, X_3^{\pm1}]$ as the Chekhov-Fock square root algebra
associated to the ideal triangulation in Figure \ref{fgggg3} where $X_i$ corresponds to $e_i$ for $i=1,2,3$.
\end{proof}

Let $V$ be an $n$ dimensional vector space over the complex field with basis
$w_0, w_1,\dots, w_{n-1}$. We can define a representation
$\rho_{r_1,r_2,r_3}: \mathbb{C}_{q}[X_1^{\pm1}, X_2^{\pm1}, X_3^{\pm1}] \rightarrow End(V)$ such that
\begin{equation}
\begin{split}
&\rho_{r_1,r_2,r_3}(X_1)w_i = r_1 q^{i} w_i,\\
&\rho_{r_1,r_2,r_3}(X_2)w_i = r_2 q^{-i} w_{i+1},\\
&\rho_{r_1,r_2,r_3}(X_3)w_i = r_3 w_{i-1},
\end{split}
\end{equation}
where $r_1,r_2,r_3$ are nonzero complex numbers. We can get $\rho_{r_1,r_2,r_3}([X_1 X_2 X_3]) = r_1 r_2 r_3 q^{1/2} Id_V$.
\begin{lmm}
For any three nonzero complex numbers $r_1,r_2,r_3$,  we have $\rho_{r_1,r_2,r_3}$ is an irreducible representation of
$\mathbb{C}_{q}[X_1^{\pm1}, X_2^{\pm1}, X_3^{\pm1}]$. Furthermore every irreducible representation of $\mathbb{C}_{q}[X_1^{\pm1}, X_2^{\pm1}, X_3^{\pm1}]$
is isomorphic to a representation $\rho_{r_1,r_2,r_3}$, and
 $\rho_{r_1,r_2,r_3} \simeq \rho_{s_1,s_2,s_3}$ if and only if $$r_1^{n} = s_1^{n}, r_2^{n} = s_2^{n}, r_3^{n} = s_3^{n}, r_1r_2r_3 = s_1s_2s_3.$$
\end{lmm}

For any $\gamma \in \mathcal{X}_{SL(2,\mathbb{C})}(S_{1,1})$ and a nonzero complex number $p$ such that $T_n(p) = - {\rm{Trace}}\gamma(P)$ where $P$ is the loop going around the only puncture in $S_{1,1}$, let $t_i = -{\rm{Trace}}\gamma(K_i)$ for $i= 1,2,3$.
According to \cite{NT}, we have
$$T_n(p) = -t_1t_2t_3 - t_1^{2} - t_2^{2} - t_3^{2} + 2.$$

\begin{lmm}\label{lmm3.10}
Let $x, y$ be two indeterminates such that $x y = q^{-2} y x$, then $T_n(x + x^{-1} + y) = x^{-n} + x^{n} + y^n$ for $n\geq1$.
\end{lmm}

For any given character $[\gamma] \in \mathcal{X}_{SL(2,\mathbb{C})}(S_{1,1})$, the following Theorem offers a representation of $SK_{q^{1/2}}(S_{1,1})$ whose classical shadow is $[\gamma]$.

\begin{thm}\label{thm3.5}
With the above notation, we have $\rho_{r_1,r_2,r_3} F$ is a representation of $SK_{q^{1/2}}(S_{1,1})$. The classical shadow of $\rho_{r_1,r_2,r_3} F$ is $\gamma$
 and $\rho_{r_1,r_2,r_3} F(P) = pId_V$ if and only if we have the following equations:
 \begin{equation}\label{eq3.6}
 \begin{split}
 &r_2^{n} r_3^{n} + r_2^{-n} r_3^{-n} + r_2^{n} r_3^{-n} =  -t_1,\\
 &r_3^{n} r_1^{n} + r_3^{-n} r_1^{-n} + r_3^{n} r_1^{-n} =  -t_2,\\
 &r_1^{n} r_2^{n} + r_1^{-n} r_2^{-n} + r_1^{n} r_2^{-n} =  -t_3,\\
 &r_1^{2} r_2^{2} r_3^{2} q + r_1^{-2} r_2^{-2} r_3^{-2} q^{-1} =  p.
 \end{split}
 \end{equation}
\end{thm}

 \begin{proof}
It is easy to see
\begin{equation*}
\begin{split}
[X_2 X_3][X_2 X_3^{-1}] = q^{-2} [X_2 X_3^{-1}][X_2 X_3],\\
[X_3 X_1][X_3 X_1^{-1}] = q^{-2} [X_3 X_1^{-1}][X_3 X_1],\\
 [X_1 X_2][X_1 X_2^{-1}] = q^{-2} [X_1 X_2^{-1}][X_1 X_2].
\end{split}
\end{equation*}

From Lemma \ref{lmm3.10},  we get
\begin{equation}\label{eq3.11}
\begin{split}
T_n(\rho_{r_1,r_2,r_3} F(K_1)) &= T_n(\rho_{r_1,r_2,r_3}([X_2 X_3] + [X_2^{-1}X_3^{-1}] + [X_2 X_3^{-1}]))\\
&=  \rho_{r_1,r_2,r_3}(T_n([X_2 X_3] + [X_2^{-1}X_3^{-1}] + [X_2 X_3^{-1}]))\\
&= \rho_{r_1,r_2,r_3}([X_2^{n} X_3^{n}] + [X_2^{-n}X_3^{-n}] + [X_2^{n} X_3^{-n}])\\
&= \rho_{r_1,r_2,r_3}(-(X_2^{n} X_3^{n} + X_2^{-n}X_3^{-n} + X_2^{n} X_3^{-n}))\\
&= -(r_2^{n} r_3^{n} + r_2^{-n} r_3^{-n} + r_2^{n} r_3^{-n})Id_{V}.
\end{split}
\end{equation}
Similarly we can get
\begin{equation}\label{eq3.12}
\begin{split}
T_n(\rho_{r_1,r_2,r_3} F(K_2)) = -(r_3^{n} r_1^{n} + r_3^{-n} r_1^{-n} + r_3^{n} r_1^{-n})Id_{V},\\
T_n(\rho_{r_1,r_2,r_3} F(K_3)) = -(r_1^{n} r_2^{n} + r_1^{-n} r_2^{-n} + r_1^{n} r_2^{-n})Id_{V}.
\end{split}
\end{equation}
And
\begin{equation}\label{eq3.13}
\begin{split}
\rho_{r_1,r_2,r_3} F(P) &= \rho_{r_1,r_2,r_3}([X_1^{2}X_2^{2}X_1^{2}] + [X_1^{-2}X_2^{-2}X_1^{-2}])\\
&= \rho_{r_1,r_2,r_3}([X_1^{2}X_2^{2}X_1^{2}]) + \rho_{r_1,r_2,r_3}([X_1^{-2}X_2^{-2}X_1^{-2}])\\
&= (r_1^{2}r_2^{2}r_3^{2}q + r_1^{-2}r_2^{-2}r_3^{-2}q^{-1})Id_{V}.
\end{split}
\end{equation}

From equations \ref{eq3.11},\ref{eq3.12},\ref{eq3.13} and the fact that $K_1,K_2,K_3$ generate the algebra $SK_{-1}(S_{1,1})$, we can get the
conclusions in Theorem \ref{thm3.5}.

 \end{proof}

\begin{rem}
At first glance, it seems like, in equation \ref{eq3.6}, we may not be able to get solutions, but actually the forth one is a consequence of first three equations because
we have the relation $T_n(p) = -t_1t_2t_3 - t_1^{2} - t_2^{2} - t_3^{2} + 2$.
In fact, to get solutions, we only need to solve equations:
\begin{equation*}
 \begin{split}
 &y z + y^{-1} z^{-1} + y z^{-1} =  -t_1,\\
 &z x + z^{-1} x^{-1} + z x^{-1} =  -t_2,\\
 &x y + x^{-1} y^{-1} + x y^{-1} =  -t_3.
 \end{split}
 \end{equation*}

\end{rem}

Let $Y_i = X_i^{2}$ for $i= 1, 2, 3$, then $$Y_1Y_2 = q^4Y_2Y_1, Y_2Y_3 = q^4 Y_3Y_2, Y_3Y_1 = q^4Y_1Y_3, Y_iY_i^{-1} = Y_i^{-1}Y_i = 1.$$
The subalgebra of $\mathbb{C}_{q}[X_1^{\pm1}, X_2^{\pm1}, X_3^{\pm1}]$ generated by $Y_1^{\pm1}, Y_2^{\pm1}, Y_3^{\pm1}$ is  $\mathbb{C}_{q^4}[Y_1^{\pm1}, Y_2^{\pm1}, Y_3^{\pm1}]$.

Here we  recall a lemma from \cite{BW5} for irreducible representations for $\mathbb{C}_{q^4}[Y_1^{\pm1}, Y_2^{\pm1}, Y_3^{\pm1}]$.
Let $V$ be an $n$ dimensional vector space over  $\mathbb{C}$ with basis $w_0, w_1,\dots, w_{n-1}$. For any three nonzero complex numbers
$y_1, y_2, y_3$, define $\rho_{y_1,y_2,y_3}:\mathbb{C}_{q^4}[Y_1^{\pm1}, Y_2^{\pm1}, Y_3^{\pm1}] \rightarrow {\rm{End}}(V)$ such that
\begin{equation}\label{eq2.17}
\begin{split}
\rho_{y_1,y_2,y_3}(Y_1)(w_i) &= y_1 q^{4i}w_i,\\
\rho_{y_1,y_2,y_3}(Y_2)(w_i) &= y_2 q^{-2i}w_{i+1},\\
\rho_{y_1,y_2,y_3}(Y_3)(w_i) &= y_3 q^{-2i}w_{i-1}.\\
\end{split}
\end{equation}

\begin{lmm}[\cite{BW5}]\label{lmm2.9}
(1) For any three nonzero complex numbers $y_1, y_2, y_3$, the representation $\rho_{y_1,y_2,y_3}$ is irreducible.

(2) Every irreducible representation of $\mathbb{C}_{q^4}[Y_1^{\pm1}, Y_2^{\pm1}, Y_3^{\pm1}]$ is isomorphic to a representation  $\rho_{y_1,y_2,y_3}$.

(3) For $\rho_{y_1,y_2,y_3}$ and $\rho_{y_1^{'},y_2^{'},y_3^{'}}$, they are isomorphic if and only if
$y_1^{n} = (y^{'})^n, y_2^{n} = (y_2^{'})^n, y_3^{n} = (y_3^{'})^n $ and $y_1y_2y_3= y_1^{'}y_2^{'}y_3^{'}$.
\end{lmm}

\section{Calculation of intertwiners for the closed torus }\label{seccece}


\subsection{Construction of intertwiners for the closed torus}
The mapping class group of  the closed torus is $SL(2,\mathbb{Z})$ \cite{BM}.
For any $A=
\begin{pmatrix}
  a & b\\
  c & d
\end{pmatrix} \in SL(2,\mathbb{Z})$, we hope to find invariant characters under $A$.
For a $[\gamma]\in \mathcal{X}_{SL(2,\mathbb{C})}(T^2)$, we choose a representative with
$\gamma(\alpha)$=
$\begin{pmatrix}
  \lambda_{1} & 0\\
  0 & \lambda_{1}^{-1}
\end{pmatrix}$ and
$\gamma(\beta)$=
$\begin{pmatrix}
  \lambda_{2} & 0\\
  0 & \lambda_{2}^{-1}
\end{pmatrix}$ where  $\alpha$ and $\beta$  denote  loops $(1,0)$ and $(0,1)$ in $\mathbb{R}^2/\mathbb{Z}^2$ respectively. We have
$[\gamma]$ is invariant under $A$ if and only if ${\rm{Trace}}(\gamma(A_{*}(z))) = {\rm{Trace}}(\gamma(z))$ for all $z\in \pi_{1}(T^2)$.
For any $(k_1, k_2)\in\mathbb{Z}\oplus\mathbb{Z}$, we have
\begin{equation*}
\begin{split}
A_*(k_1\alpha+ k_2\beta) = (k_1, k_2)
\begin{pmatrix}
  a & b\\
  c & d
\end{pmatrix}
=(k_1 a+ k_2 c, k_1 b+ k_2 d)
=(k_1 a+ k_2 c)\alpha + (k_1 b+ k_2 d)\beta,
\end{split}
\end{equation*}
\begin{equation*}
\begin{split}
\gamma (A_*(k_1\alpha+ k_2\beta))
=&\gamma[(k_1 a+ k_2 c)\alpha + (k_1 b+ k_2 d)\beta]
=\begin{pmatrix}
  \lambda_{1} & 0\\
  0 & \lambda_{1}^{-1}
\end{pmatrix}^{k_1 a+ k_2 c}
\begin{pmatrix}
  \lambda_{2} & 0\\
  0 & \lambda_{2}^{-1}
\end{pmatrix}^{k_1 b+ k_2 d}\\
=&\begin{pmatrix}
  \lambda_{1}^{k_1 a+ k_2 c}\lambda_2^{k_1 b+ k_2 d} & 0\\
  0 & \lambda_{1}^{-(k_1 a+ k_2 c)}\lambda_{2}^{-(k_1 b+ k_2 d)}
\end{pmatrix},
\end{split}
\end{equation*}
\begin{equation*}
\gamma(k_1\alpha+ k_2\beta)=
\begin{pmatrix}
  \lambda_{1} & 0\\
  0 & \lambda_{1}^{-1}
\end{pmatrix}^{k_1}
\begin{pmatrix}
  \lambda_{2} & 0\\
  0 & \lambda_{2}^{-1}
\end{pmatrix}^{k_2}\\
=\begin{pmatrix}
  \lambda_{1}^{k_1}\lambda_2^{k_2} & 0\\
  0 & \lambda_{1}^{-k_1}\lambda_{2}^{-k_2}
\end{pmatrix}.
\end{equation*}
Then it is easy to show that $[\gamma]$ is $A$-invariant if and only if $\lambda_1=\lambda_1^{a}\lambda_2^{b}, \lambda_2=\lambda_1^{c}\lambda_2^{d}$ or
$\lambda_1=\lambda_1^{-a}\lambda_2^{-b}, \lambda_2=\lambda_1^{-c}\lambda_2^{-d}$.

And
$A$ also induces two algebra isomorphisms $F_{A,+}$ and $F_{A,-}$  from $\mathbb{C}[X^{\pm1},Y^{\pm1}]_{q^{1/2}}$ to itself defined by
$$F_{A,+}(\theta_{(i,j)})=\theta_{(i,j)A}, F_{A,-}(\theta_{(i,j)})=\theta_{(-i,-j)A}.$$
 Then $F_{A,+}$ and $F_{A,-}$ are well-defined becasue
$$\theta_{(i,j)A}\theta_{(k,l)A} = q^{1/2\begin{bmatrix}
(i,j)A\\
 (k,l)A
\end{bmatrix}} \theta_{(i+k, j+l)A}
=q^{1/2\begin{bmatrix}
i & j\\
k & l
\end{bmatrix}[A]} \theta_{(i+k, j+l)A}
=q^{1/2\begin{bmatrix}
i & j\\
k & l
\end{bmatrix}} \theta_{(i+k, j+l)A},
$$
and similarly
$$\theta_{(-i,-j)A}\theta_{(-k,-l)A}
=q^{1/2\begin{bmatrix}
i & j\\
k & l
\end{bmatrix}} \theta_{(-i-k, -j-l)A}.
$$

In section \ref{seeee}, we know there is an embedding
$$G_{q^{1/2}}: SK_{q^{1/2}}(T^{2})\rightarrow \mathbb{C}[X^{\pm 1},Y^{\pm 1}]_{q^{1/2}}.$$
For the following discussion we will  omit the subscript for $G_{q^{1/2}}$ when there is no confusion.

\begin{lmm}
The following diagram is commutative:
\begin{equation*}
\xymatrix{
&SK_{q^{1/2}}(T^{2})\ar[r]^{G}\ar[d]^{A_{\sharp}}&\mathbb{C}[X^{\pm 1},Y^{\pm 1}]_{q^{1/2}}\ar[d]^{F_A}\\
&SK_{q^{1/2}}(T^{2})\ar[r]^{G}&\mathbb{C}[X^{\pm 1},Y^{\pm 1}]_{q^{1/2}}
}
\end{equation*}
where $F_A$ is $F_{A,+}$ or $F_{A,-}$.
\end{lmm}
\begin{proof}
We only prove the case when $F_A = F_{A,+}$.

First we show $A_{\sharp}((k,l)_{T})=((k,l)A)_{T}$.
Assume $\text{gcd}(k,l)=j$ and $k=k^{'}j, l= l^{'}j$, then we have $$A_{\sharp}((k,l)_T)= A_{\sharp}(T_j((k^{'},l^{'})))
= T_j(A_{\sharp}((k^{'},l^{'}))) = T_j((k^{'},l^{'})A) = T_j(ak^{'}+cl^{'}, bk^{'}+dl^{'}).$$
There exist integers $u,v$ such that $uk+vl = j$, then
$ \begin{bmatrix}
k & l\\
-v & u
\end{bmatrix} = j$ and $\det(\begin{pmatrix}
k & l\\
-v & u
\end{pmatrix}A) = j$. Then
$\det(\begin{pmatrix}
k & l\\
-v & u
\end{pmatrix}A) =
\begin{bmatrix}
ak+cl & bk+dl\\
-v^{'} & u^{'}
\end{bmatrix} = (ak+cl)u^{'} + (bk+dl)v^{'} =  j$. We also have $j|(ak+cl), j|(bk+dl)$. Thus $\text{gcd}(ak+cl,bk+dl) =j$ and $ak+cl = j(ak^{'}+cl^{'}), bk+dl = j(bk^{'}+dl^{'})$. Then
$$A_{\sharp}((k,l)_T)= T_j(ak^{'}+cl^{'}, bk^{'}+dl^{'}) =  (ak+cl,bk+dl)_{T} = ((k,l)A)_T.$$

So
$G A_{\sharp}((k,l)_T) = G(((k,l)A)_T) = \theta_{(k,l)A} + \theta_{(k,l)A}^{-1}$ and $F_AG((k,l)_T)= F_A(\theta_{(k,l)}+ \theta_{(k,l)}^{-1}) = \theta_{(k,l)A} + \theta_{(k,l)A}^{-1}$. We have $G A_{\sharp} = F_AG$ because all $(k,l)_T$ span the skein algebra.
\end{proof}

\begin{equation}\label{EQ2.6}
\xymatrix{
&SK_{q^{1/2}}(T^{2})\ar[r]^{G}\ar[d]^{A_{\sharp}}&\mathbb{C}[X^{\pm 1},Y^{\pm 1}]_{q^{1/2}}\ar[r]^{\rho_{u,v}}\ar[d]^{F_A}&End(V)\ar@{.>}[d]^{G_{\Lambda}}\\
&SK_{q^{1/2}}(T^{2})\ar[r]^{G}&\mathbb{C}[X^{\pm 1},Y^{\pm 1}]_{q^{1/2}}\ar[r]^{\rho_{u,v}}&End(V)
}
\end{equation}

 The following two Theorems give the intertwiners for the closed surface for all the diffeomorphisms. We will give explicit formulas for these intertwiners and their Trace in the following subsections.

\begin{thm}\label{thm2.3}
In the diagram \ref{EQ2.6}, suppose
$A=\begin{pmatrix}
a & b\\
c & d
\end{pmatrix} \in SL(2,\mathbb{Z})$ and $F_A = F_{A,+}$. Let $[\gamma]\in \mathcal{X}_{SL(2,\mathbb{C})}(T^2)$ with
$\gamma(\alpha)$=
$\begin{pmatrix}
  \lambda_{1} & 0\\
  0 & \lambda_{1}^{-1}
\end{pmatrix}$ and 
$\gamma(\beta)$=
$\begin{pmatrix}
  \lambda_{2} & 0\\
  0 & \lambda_{2}^{-1}
\end{pmatrix}$ where $\lambda_1=\lambda_1^{a}\lambda_2^{b}, \lambda_2=\lambda_1^{c}\lambda_2^{d}$, and let $u,v$ be two complex numbers such that $u^{n}=-\lambda_{1}$, $v^{n}=-\lambda_{2}$. We have the following conclusions:

(a) $[\gamma]$ is invariant under $A$;

(b) the classical shadow of $\rho_{u,v}G$ is $[\gamma]$;

(c) $\rho_{u,v}F_A\simeq \rho_{u,v}$.

(d) From (c), we know there exists an  interwiner $\Lambda_{n,+}$ such that $\rho_{u,v} F_A(Z) = \Lambda_{n,+} \rho_{u,v}(Z) \Lambda^{-1}_{n,+}$ for
all $Z\in \mathbb{C}[X^{\pm 1},Y^{\pm 1}]_{q^{1/2}}$. Then this intertwiner induces an intertwiner between two irreducible representations of the skein algebra.

\end{thm}

\begin{proof}
(a) and (b) are already shown in the previous discussion.

To prove (c), we have
\begin{equation*}
\begin{split}
 &\rho_{u,v}F_A(X^n) = \rho_{u,v}F_A(\theta_{(n,0)}) = \rho_{u,v}(\theta_{(n,0)A})\\
= &\rho_{u,v}(\theta_{(na,nb)}) = \rho_{u,v}((-1)^{ab}X^{na}Y^{nb})\\
=&(-1)^{ab}u^{na}v^{nb}Id_{V} = (-1)^{ab}(-\lambda_1)^{a}(-\lambda_2)^{b}Id_{V}\\
=&(-1)^{ab+a+b}\lambda_1^{a}\lambda_2^{b}Id_{V} = -\lambda_1 Id_V = u^n Id_V.
\end{split}
\end{equation*}
Similarly we can show $\rho_{u,v}F_A(Y^n) = v^n Id_V$, thus $\rho_{u,v}F_A\simeq \rho_{u,v}$.

(d) If $\lambda_1\neq \pm1$ or $\lambda_1\neq \pm1$,   Theorem \ref{thm2.2} implies that $\Lambda_{n,+}$ itself is the intertwiner between two irreducible representations of the skein algebra.
If $\lambda_1 = \pm1$ and $\lambda_1= \pm1$,  Theorem \ref{thm2.2} implies that $V$ has only two irreducible subrepresntations $V_1$, $V_2$ and
$\dim(V_1) = (n+1)/2, \dim(V_2) = (n-1)/2$. We have $\Lambda_{n,+}(V_1)$ is an irreducible subrepresentation of $V$ and $\dim(\Lambda_{n,+}(V_1)) = (n+1)/2$. Then $\Lambda_{n,+}(V_1) = V_1$.
This shows $\Lambda_{n,+}|_{V_1}$ is an intertwiner for $V_1$. Similarly $\Lambda_{n,+}|_{V_2}$ is an intertwiner for $V_2$.
\end{proof}

\begin{thm}\label{thm2.4}
In the diagram \ref{EQ2.6}, suppose
$A=\begin{pmatrix}
a & b\\
c & d
\end{pmatrix} \in SL(2,\mathbb{Z})$ and $F_{A} = F_{A,-}$. Let $[\gamma]\in \mathcal{X}_{SL(2,\mathbb{C})}(T^2)$  with
$\gamma(\alpha)$=
$\begin{pmatrix}
  \lambda_{1} & 0\\
  0 & \lambda_{1}^{-1}
\end{pmatrix}$ and
$\gamma(\beta)$=
$\begin{pmatrix}
  \lambda_{2} & 0\\
  0 & \lambda_{2}^{-1}
\end{pmatrix}$ where $\lambda_1=\lambda_1^{-a}\lambda_2^{-b}, \lambda_2=\lambda_1^{-c}\lambda_2^{-d}$, and let $u,v$ be two complex numbers such that $u^{n}=-\lambda_{1}^{-1}$, $v^{n}=-\lambda_{2}^{-1}$. We have the following conclusions:

(a) $[\gamma]$ is invariant under $A$;

(b) the classical shadow of $\rho_{u,v}G$ is $[\gamma]$;

(c) $\rho_{u,v}F_A\simeq \rho_{u,v}$.

(d) From (c), we know there exists an  interwiner $\Lambda_{n,-}$ such that $\rho_{u,v} F_A(Z) = \Lambda_{n,-} \rho_{u,v}(Z) \Lambda^{-1}_{n,-}$ for
all $Z\in \mathbb{C}[X^{\pm 1},Y^{\pm 1}]_{q^{1/2}}$. Then this intertwiner induces an intertwiner between two irreducible representations of the skein algebra.

\end{thm}

\begin{proof}
The proof is the same as in Theorem \ref{thm2.3}.
\end{proof}

Note that a rescaling of  the intertwiner $\Lambda_{n,+}$  in Theorem \ref{thm2.3} such that $|\det(\Lambda_{n,+})| = 1$ makes
$|\text{Trace} \Lambda_{n,+}|$ independent of the choice of $u$ and $v$. The same thing holds for the intertwiner in Theorem \ref{thm2.4}.

For the following discussion, we always require $F_A$ to be $F_{A,+}$ unless specified otherwise (parallel results hold for $F_{A,-}$).
From the above discussion, we know there exists an intertwiner $\Lambda_n\in \rm{End}(V)$ such that the diagram  \ref{EQ2.6} commutes, where
$G_{\Lambda}(B) = \Lambda_n B \Lambda_n^{-1},\forall B\in \rm{End}(V)$. Next we are going to find an intertwiner  $\overline{\Lambda}_n$ under the assumption in Theorem \ref{thm2.3}.

\subsection{Calculation for intertwiners}
Under the assumption of Theorem \ref{thm2.3}, we have $\rho_{u,v}F_A\simeq \rho_{u,v}$. For any $a\in V$ and $Z\in \mathbb{C}[X^{\pm 1},Y^{\pm 1}]_{q^{1/2}}$,
we use $Z\cdot a$ and  $Z\star a$ to denote $\rho_{u,v}(Z)(a)$ and $\rho_{u,v}F_A(Z)(a)$ respectively. Then we are trying to find $\overline{\Lambda}_n\in \rm{End}(V)$
such that $\overline{\Lambda}_n(X\cdot a) = X\star(\overline{\Lambda}_n (a))$ and $\overline{\Lambda}_n(Y\cdot a) = Y\star(\overline{\Lambda}_n (a))$ for all $a\in V$.

\begin{rem}\label{rem4.1}
Assume $\text{gcd}(b,n)= m$ and $n=n^{'}m$. There exist two integers $r,s$ such that $br+sn=m$. Then we have
\begin{equation}
\begin{split}
&( v^{n^{'}b}u^{n^{'}(a-1)}q^{ab(n^{'})^2/2} )^m = v^{mn^{'}b}u^{mn^{'}(a-1)}q^{abm(n^{'})^2/2} \\
= &v^{nb}u^{n(a-1)}q^{abnn^{'}/2} = (-\lambda_2)^{b}(-\lambda_1)^{a-1}(-1)^{abn^{'}}\\
=&(-1)^{ab+b+a-1}\lambda_2^{b}\lambda_1^{a-1}=1,
\end{split}
\end{equation}
and $q^{an^{'}}$ is a primitive $m$-root of unity. 
Then there exists a unique  integer $0\leq k_0\leq m-1$ such that $( v^{n^{'}b}u^{n^{'}(a-1)}q^{ab(n^{'})^2/2} )q^{an^{'}k_0} = 1$ and
 $( v^{n^{'}b}u^{n^{'}(a-1)}q^{ab(n^{'})^2/2} )q^{an^{'}k}\neq 1$ for $k\neq k_0, 0\leq k\leq m-1$.
We set $r_{k_0} = 1$ and $r_{k} = 0$  for $k\neq k_0, 0\leq k\leq m-1$, and define $r_{k+tb} = r_{k} v^{tb}u^{t(a-1)} q^{a(tk+bt^2/2)}, \forall 0\leq k\leq m-1, t\in \mathbb{Z}$,
where we consider all indices modulo $n$. Since gcd$(b,n) = m$ and $n= mn^{'}$, we can reach all the indices. It is an easy check that $r_{k_1 +t_1 b} = r_{k_2 +t_2 b}$ if
  $k_1 +t_1 b \equiv k_2 +t_2 b \;(\mod\text{ }n)$. Then $r_k$ is well-defined for each $0\leq k\leq n-1$.

It is easy to check that we have $r_{k+tb} = r_kv^{tb}u^{t(a-1)}q^{atk+abt^2/2}$ for all $k,t\in \mathbb{Z}$. 
Actually we have
  $$r_{k_{0}+tb} = v^{tb}u^{t(a-1)} q^{a(tk_0+bt^2/2)}, \forall 0\leq t\leq n^{'}-1,$$ and all other $r_k$ are  0.
  We have $(v^b u^{a-1})^{n} = (v^n)^b (u^n)^{a-1} = (-\lambda_2)^{b} (-\lambda_1)^{a-1} = (-1)^{a+b-1}{\lambda_1}^{a-1} {\lambda_2}^{b}\\= (-1)^{a+b-1}$.
  Then we get $|r_k| =$ 0 or 1 for all $k\in \mathbb{Z}$.
  From $br+sn=m$, we get $tm = tbr + tsn$ for all $t\in \mathbb{Z}$. Then  we have
  $$r_{k_{0}+tm} = r_{k_{0}+trb} = v^{trb}u^{tr(a-1)} q^{a(trk_0+bt^2r^2/2)}, \forall 0\leq t\leq n^{'}-1,$$ and all other $r_k$ are  0.
\end{rem}

 The following Lemma offers an explicit formula for the intertwiner constructed in  Theorem \ref{thm2.3} (d).

\begin{lmm}\label{lmm4.9}
Under the assumption of Theorem \ref{thm2.3}, suppose $\overline{\Lambda}_n \in \rm{End}(V)$ and $$\overline{\Lambda}_n(e_t) = \sum_{0\leq k\leq n-1}(\overline{\Lambda}_n)_{k,t}e_k$$
for all $0\leq t\leq n-1$, where
$$(\overline{\Lambda}_n)_{k,t} = r_{k-td} (v^{(d-1)}u^{c})^t q^{c(tk-dt^2/2)}.$$
Then $\overline{\Lambda}_n$  satisfies the conditions  in Theorem \ref{thm2.3} (d).
\end{lmm}

\begin{proof}
From direct calculations, we can get $\overline{\Lambda}_n(X\cdot e_t) = X\star(\overline{\Lambda}_n (e_t))$ and $\overline{\Lambda}_n(Y\cdot e_t) = Y\star(\overline{\Lambda}_n (e_t))$ for all $0\leq t\leq n-1$.
\end{proof}

We have $(v^{d-1} u^{c})^{n} = (v^n)^{d-1} (u^n)^{c} = (-\lambda_2)^{d-1} (-\lambda_1)^{c} = (-1)^{c+d-1}{\lambda_1}^{c} {\lambda_2}^{d-1} = (-1)^{c+d-1}$. Then we can get $|(\overline{\Lambda}_n)_{k,t}| = 0 \text{ or } 1$.
We have $(\overline{\Lambda}_n)_{k,t} = 0$ if and only if  $r_{k-td} = 0$. Then it is easy to show that $(\overline{\Lambda}_n)_{ld + km + k_0, l+tm}, 0\leq l\leq m-1, 0\leq k,t\leq n^{'} - 1$,
are the only nonzero entries.

For each $0\leq l\leq m-1$, we define an $n^{'}\times n^{'}$ matrix $B^{l}$ such that $(B^{l})_{k,t} = (\overline{\Lambda}_n)_{ld + km + k_0, l+tm}$ for all $0\leq k,t\leq n^{'} - 1$.
Then by Laplace expansion, we know $|\det(\overline{\Lambda}_n)| = \prod_{0\leq l \leq m} |\det(B^{l})|$.

From pure calculations, we can get $|\det(B^{l})| = (n^{'})^{n^{'}/2}$.
%
Then we have $$|\det(\overline{\Lambda}_n)| = \prod_{0\leq l \leq m} |\det(B^{l})| = ((n^{'})^{n^{'}/2})^m = (n^{'})^{mn^{'}/2} = (n^{'})^{n/2},$$ furthermore
$|\det((n^{'})^{-1/2}\overline{\Lambda}_n)| = 1$.

\begin{rem}
If $\widetilde{\Lambda}_n$ is the intertwiner in Theorem \ref{thm2.4}, and we still suppose $\text{gcd}(b,n) = m, br + sn =m, n=n^{'}m$.  We have
$$r_{k_0 - tm} = (v^{-b}u^{-a-1})^{tr}q^{-trak_0 + abt^2r^2/2},\forall t\in\mathbb{Z}; \text{ }r_k =0,\text{ otherwise}$$
 $$(\widetilde{\Lambda}_n)_{k,t} = r_{k+td}(v^{-d-1}u^{-c})^t q^{-tck - cdt^2/2},\forall 0\leq k,t\leq n-1,$$
 where $0\leq k_0 \leq m-1$ such that
 $$(v^{-b}u^{-a-1})^{n^{'}}q^{ab(n^{'})^2/2}q^{-n^{'}ak_0} = 1.$$
 Also we can get $|\det(\widetilde{\Lambda}_n)| = (n^{'})^{n/2}$.
\end{rem}

\subsection{On the trace of intertwiners}

Bonahon-Wong-Yang only formulated the conjecture when the mapping tori are hyperbolic. So they
 considered surfaces with negative Euler characteristic because the mapping tori for the closed torus can never be hyperbolic. 
 Since the simplicial volume of mapping tori for the closed torus is zero, see page 380 \cite{BM}, we expect the corresponding  limit to be zero. 
In Theorem \ref{thm4.1}, we can show the limit superior is zero for any diffeomorphism. But the limits are not zero for some cases, see Example \ref{exa4.1}.
Some diffeomorphisms even do not have invariant characters that live in the Azumaya  locus, but the  intertwiners in Theorems \ref{thm2.3} and \ref{thm2.4} are very close to  intertwiners  constructed in
\cite{BW5}.


When we consider the intertwiners in Theorems \ref{thm2.3} and \ref{thm2.4}, we fix the mapping class $A$ and the $A$-invariant character $[\gamma]$.
In this subsection we will use $(l,s)$ to denote $\text{gcd}(l,s)$ for any two integers $l,s$.

\begin{thm}\label{thm2.5}
If we require $|\det(\Lambda_n)| = 1$ for the intertwiner in Theorem \ref{thm2.3}, then $|{\rm{Trace}}\Lambda_n|\leq  n^{3/2}$.
\end{thm}
\begin{proof}
Since any two intertwiners in Theorem \ref{thm2.3} are different by a scalar multiplication or by conjugation and we require $|\det(\Lambda_n)| = 1$, the absolute value $|\rm{Trace}\Lambda_n|$ is independent of the choice of intertwiners. Let $\Lambda_n = (n^{'})^{-1/2}\overline{\Lambda}_n$, then $|\det(\Lambda_n)| = 1$.
Since $|(\overline{\Lambda}_n)_{k,t}| =0 \text{ or } 1$ for all $0\leq k,t\leq n-1$ and each row has exactly $n^{'}$ nonzero entries, we have the absolute value of every eigenvalue
of $\overline{\Lambda}_n$ is not more than $n^{'}$. Then $$|\text{Trace}(\Lambda_n)|=|\text{Trace}((n^{'})^{-1/2}\overline{\Lambda}_n)| \leq (n^{'})^{-1/2}(nn^{'}) = (n^{'})^{1/2}n \leq  n^{3/2}.$$
\end{proof}

\begin{thm}
If we require $|\det(\Lambda_n)| = 1$ for the intertwiner in Theorem \ref{thm2.4}, then $|{\rm{Trace}}\Lambda_n|\leq  n^{3/2}$.
\end{thm}
\begin{proof}
It is similar with the proof for Theorem \ref{thm2.5}.
\end{proof}

\begin{lmm}
Let $k$ be any  integer, then we have
\begin{equation*}
|\sum_{0\leq t\leq n-1}(-q^{\frac{1}{2}})^{kt^2}| = \sqrt{(k, n)n}.
\end{equation*}
 Recall that $q^{1/2}$ is a primitive $n$-root of $-1$.
\end{lmm}
\begin{proof}
In \cite{HTM}, they proved this result for $k=2$. Using the same trick, we can prove this generalized lemma.

\end{proof}

In the following of this section, we always assume  $q^{1/2} = e^{\pi i/n}$ unless especially specified.
Next we are going to calculate $\text{Trace}\Lambda_n$, where $\Lambda_n$ is the intertwiner in
Theorem \ref{thm2.3} or \ref{thm2.4} with $|\det\Lambda_n| = 1$.
First we give detailed  discussion on the invariant character. Recall that for any $A=
\begin{pmatrix}
  a & b\\
  c & d
\end{pmatrix} \in SL(2,\mathbb{Z})$ and a character $[\gamma]\in \mathcal{X}_{SL(2,\mathbb{C})}(T^2)$ with
$\gamma(\alpha)$=
$\begin{pmatrix}
  \lambda_{1} & 0\\
  0 & \lambda_{1}^{-1}
\end{pmatrix}$ and
$\gamma(\beta)$=
$\begin{pmatrix}
  \lambda_{2} & 0\\
  0 & \lambda_{2}^{-1}
\end{pmatrix}$, we have $[\gamma]$ is  $A$-invariant if and only if $1=\lambda_1^{a-1}\lambda_2^{b}, 1=\lambda_1^{c}\lambda_2^{d-1}$ or
$1=\lambda_1^{a+1}\lambda_2^{b}, 1=\lambda_1^{c}\lambda_2^{d+1}$.

\begin{rem}\label{rem3.2}
We will provide a detailed discussion only for the case when  $1=\lambda_1^{a-1}\lambda_2^{b}, 1=\lambda_1^{c}\lambda_2^{d-1}$.
Suppose $\lambda_1 = \alpha_1 e^{i\theta_1}, \lambda_2 = \alpha_2 e^{i\theta_2}$, then we can get equations:
\begin{equation*}
\begin{cases}
1=\alpha_1^{a-1}\alpha_2^{b}\\
1=\alpha_1^{c}\alpha_2^{d-1}
\end{cases}
\end{equation*}
\begin{equation}\label{eqq3.10}
\begin{cases}
(a-1)\theta_1 + b\theta_2 = 2k_1 \pi\\
c\theta_1 + (d-1)\theta_2 = 2k_2 \pi
\end{cases}.
\end{equation}
Since $u^n = -\lambda_1, v^n = -\lambda_2$, we can suppose $u = -\alpha_1^{\frac{1}{n}}e^{\frac{i\theta_1}{n}}q^{r_1},
v = -\alpha_2^{\frac{1}{n}}e^{\frac{i\theta_2}{n}}q^{r_2}$
where both $r_1$ and $ r_2$ are integers. Then we have
\begin{equation}\label{eqq3.11}
\begin{split}
u^{a-1}v^b = (-1)^{a+b-1}q^{r_1(a-1) + r_2b}e^{\frac{i}{n}((a-1)\theta_1 + b\theta_2)} = (-1)^{a+b-1}q^{r_1(a-1) + r_2b + k_1},\\
u^{c}v^{d-1} = (-1)^{c+d-1}q^{r_1c + r_2(d-1)}e^{\frac{i}{n}(c\theta_1 + (d-1)\theta_2)} = (-1)^{c+d-1}q^{r_1c + r_2(d-1) + k_2}.
\end{split}
\end{equation}
Define $s_1 = r_1(a-1) + r_2b + k_1, s_2 = r_1c + r_2(d-1) + k_2$. Then, we have $u^{a-1}v^b = (-1)^{a+b-1}q^{s_1}, u^{c}v^{d-1}= (-1)^{c+d-1}q^{s_2}$.

From $1=\lambda_1^{a-1}\lambda_2^{b}, 1=\lambda_1^{c}\lambda_2^{d-1}$,  we can get $1=\lambda_1^{(a-1)c}\lambda_2^{bc},
1=\lambda_1^{(a-1)c}\lambda_2^{(a-1)(d-1)}$. Thus we have
$$\lambda_2^{bc} = \lambda_2^{(a-1)(d-1)} = \lambda_2^{ad - (a+d) +1}\Rightarrow 1 = \lambda_2^{ad - bc - (a+d) +1} = \lambda_2^{2 - (a+d)}.$$
If $a+d\neq 2$,
 we have $\lambda_2$ is a root of unity with $\lambda_2^{2 - (a+d)}=1$. Similarly we can show $\lambda_1$ is also a root of unity,   with  $\lambda_1^{2 - (a+d)}=1$,
 under the assumption $a+d\neq 2$.



We look at the case when $(b,n) = 1$, and suppose $br + sn = 1$. Then we have
$$q^{br} = q, q^{\frac{br}{2}} = (-1)^sq^{\frac{1}{2}}, (-1)^s = (-1)^{br + 1}.$$
When $(b,n) = 1$, we can choose $k_0 = 1$ and set $r_{k_0} = r_0 = 1$. Then we have
 $$r_{tb} = v^{tb}u^{t(a-1)} q^{abt^2/2}, \forall t\in \mathbb{Z}.$$
 For any $k\in\mathbb{Z}$, we have $k = krb + ksn$. Then we have
 $$r_k = r_{krb} = v^{krb}u^{kr(a-1)} q^{abk^2 r^2/2}.$$

From the above discussion, we know we can choose $\Lambda_n$ to be $n^{-1/2}\overline{\Lambda}_n$. We have
\begin{align*}
(\overline{\Lambda}_n)_{t,t}& = r_{t-td}(v^{d-1}u^c)^t q^{ct^2 - cdt^2/2} = (-1)^{cdt}r_{t-td}q^{s_2t} q^{ct^2 - cdt^2/2}\\
&= (-1)^{cdt}(v^b u^{a-1})^{r(t-td)}q^{abr^2(t-td)^2/2}q^{s_2t} q^{ct^2 - cdt^2/2}\\
&= (-1)^{cdt}(-1)^{abr(t-td)}q^{s_1r(t-td)}q^{abr^2(t-td)^2/2}q^{s_2t} q^{ct^2 - cdt^2/2}\\
&= (-1)^{cdt}(-1)^{abr(t-td)}   (q^{s_1r(1-d)}q^{s_2})^t  ((q^{\frac{br}{2}})^{ar(1-d)^2}q^{c - \frac{1}{2}cd})^{t^2}\\
&= (-1)^{cdt}(-1)^{abr(t-td)}   (q^{s_1r(1-d)}q^{s_2})^t  (((-1)^s q^{\frac{1}{2}})^{ar(1-d)^2}q^{c - \frac{1}{2}cd})^{t^2}\\
&= (-1)^{cdt}(-1)^{abr(t-td)}(-1)^{tsar(1-d)^2}   (q^{s_1r(1-d)}q^{s_2})^t  ( q^{\frac{1}{2}ar} q^{-ard }
q^{\frac{1}{2}ard^2} q^{c - \frac{1}{2}cd})^{t^2}\\
&= (-1)^{cdt}(-1)^{abr(t-td)}(-1)^{tsar(1-d)^2}   (q^{s_1r(1-d)}q^{s_2})^t  ( q^{\frac{1}{2}ar} q^{-r} q^{-rbc  }
q^{\frac{1}{2}rd} q^{\frac{1}{2}rbcd} q^{c - \frac{1}{2}cd})^{t^2}\\
&= (-1)^{cdt}(-1)^{abr(t-td)}(-1)^{tsar(1-d)^2}   (q^{s_1r(1-d)}q^{s_2})^t  ( q^{\frac{1}{2}ar} q^{-r} q^{-c  }
q^{\frac{1}{2}rd} ((-1)^s q^{\frac{1}{2}})^{cd} q^{c - \frac{1}{2}cd})^{t^2}\\
&= (-1)^{cdt}(-1)^{abr(t-td)}(-1)^{tsar(1-d)^2} (-1)^{tscd}   (q^{s_1r(1-d)}q^{s_2})^t  ( q^{\frac{1}{2}(a+d-2)r} )^{t^2}\\
&= (-1)^{cdt}(-1)^{abr(t-td)}(-1)^{t(br+1)ar(1-d)} (-1)^{t(br+1)cd}   (q^{s_1r(1-d)}q^{s_2})^t  ( q^{\frac{1}{2}(a+d-2)r} )^{t^2}\\
&=(-1)^{tar} (-1)^{tard} (-1)^{trd(ad+1)}   (q^{s_1r(1-d)}q^{s_2})^t  ( q^{\frac{1}{2}(a+d-2)r} )^{t^2}\\
&=(-1)^{tar}  (-1)^{trd}   (q^{s_1r(1-d)}q^{s_2})^t  ( q^{\frac{1}{2}(a+d-2)r} )^{t^2}\\
&=((-1)^{(a+d-2)r} )^t (q^{s_1r(1-d)}q^{s_2})^t  ( q^{\frac{1}{2}(a+d-2)r} )^{t^2}.\\
\end{align*}
Since $q^{s_2} = q^{rs_2b}$, we have
\begin{equation}
(\overline{\Lambda}_n)_{t,t} = ((-1)^{(a+d-2)r} )^t q^{\frac{r}{2}((a+d-2)t^2 + 2(s_1(1+d) - s_2b)t)},
\end{equation}
and
\begin{equation}\label{eqq3.13}
\text{Trace}\Lambda_n = n^{-\frac{1}{2}}\sum_{0\leq t\leq n-1} ((-1)^{(a+d-2)r} )^t q^{\frac{r}{2}((a+d-2)t^2 + 2(s_1(1+d) - s_2b)t)}.
\end{equation}

\end{rem}

\begin{rem}
Here we state the parallel results for $1=\lambda_1^{a+1}\lambda_2^{b}, 1=\lambda_1^{c}\lambda_2^{d+1}$ and $u^n = -\lambda_1^{-1}, v^n = -\lambda_2^{-1}$.

Suppose $\lambda_1 = \alpha_1 e^{i\theta_1}, \lambda_2 = \alpha_2 e^{i\theta_2}$, then we can get equations:
\begin{equation*}
\begin{cases}
1=\alpha_1^{a+1}\alpha_2^{b}\\
1=\alpha_1^{c}\alpha_2^{d+1}
\end{cases}
\end{equation*}
\begin{equation*}
\begin{cases}
(a+1)\theta_1 + b\theta_2 = 2k_1 \pi\\
c\theta_1 + (d+1)\theta_2 = 2k_2 \pi
\end{cases}.
\end{equation*}
Since $u^n = -\lambda_1^{-1}, v^n = -\lambda_2^{-1}$, we can suppose $u = -\alpha_1^{-\frac{1}{n}}e^{-\frac{i\theta_1}{n}}q^{r_1},
v = -\alpha_2^{-\frac{1}{n}}e^{-\frac{i\theta_2}{n}}q^{r_2}$
where both $r_1$ and $ r_2$ are integers. Then we have
\begin{equation*}
\begin{split}
u^{a+1}v^b = (-1)^{a+b+1}q^{r_1(a+1) + r_2b - k_1},\\
u^{c}v^{d+1} = (-1)^{c+d+1}q^{r_1c + r_2(d+1) - k_2}.
\end{split}
\end{equation*}
Similarly we set $s_1 = r_1(a+1) + r_2b - k_1, s_2 = r_1c + r_2(d+1) - k_2$, then we have $u^{a+1}v^b = (-1)^{a+b+1}q^{s_1 }, u^{c}v^{d+1} = (-1)^{c+d+1}q^{s_2}$.
If $a+d \neq -2$, we have $\alpha_1 = \alpha_2 = 1$.

For the case when $(b,n) = 1$ and $br + sn = 1$, we have
$$(\widetilde{\Lambda}_n)_{t,t} = ((-1)^{(a+d+2)r} )^t q^{\frac{r}{2}((a+d+2)t^2 + 2(s_1(1+d) - s_2b)t)},$$
and
\begin{equation}\label{eqq3.14}
\text{Trace}\Lambda_n = n^{-\frac{1}{2}}\sum_{0\leq t\leq n-1} ((-1)^{(a+d+2)r} )^t q^{\frac{r}{2}((a+d+2)t^2 + 2(s_1(1+d) - s_2b)t)}.
\end{equation}

\end{rem}

\begin{exam}\label{exa4.1}
Let $A =$
$\begin{pmatrix}
  2 & 1\\
  -7 & -3
\end{pmatrix}$.
If we try to solve $1=\lambda_1^{a-1}\lambda_2^{b}, 1=\lambda_1^{c}\lambda_2^{d-1}$,   we can get equations:
\begin{equation}
\begin{cases}
 \theta_1 + \theta_2 = 2\pi\\
-7\theta_1 -4\theta_2 = 2\pi
\end{cases}.
\end{equation}
We have $\theta_1 = -\frac{10\pi}{3}, \theta_2 = \frac{16\pi}{3}$, thus $\lambda_1 = e^{\frac{2\pi i}{3}}, \lambda_2 = e^{\frac{4\pi i}{3}}$.
So we can set $u = - e^{-\frac{10\pi i}{3n}}, v = - e^{\frac{16\pi i}{3n}}$, then we have $u^{a-1} v^b = q, u^{c}v^{d-1} = -q$.
We have $s_1 = s_2 = 1$. Since $b = 1$, we get $r = 1$. Then from equation \ref{eqq3.13}, we have
\begin{equation*}
{\rm{Trace}}\Lambda_n = n^{-\frac{1}{2}}\sum_{0\leq t\leq n-1} (-1 )^t q^{\frac{1}{2}(-3t^2 + 10t)}.
\end{equation*}
Note that when $n$ is a multiple of 3, we have ${\rm{Trace}}\Lambda_n = 0$.

\end{exam}

\begin{exam}
Let $A $
be the same matrix as above. But this time we try to
solve $1=\lambda_1^{a+1}\lambda_2^{b}, 1=\lambda_1^{c}\lambda_2^{d+1}$,
then we get $\lambda_1^{2+a+d} = \lambda_2^{2+a+d}  = 1$. Since $2+a+d = 1$, so we have $\lambda_1 = \lambda_2 = 1$.
We can set $u = v = -1$, then $s_1 = s_2 = 0$.
From equation \ref{eqq3.13}, we have
\begin{equation*}
|{\rm{Trace}}\Lambda_n| = n^{-\frac{1}{2}}|\sum_{0\leq t\leq n-1} (-1 )^t q^{\frac{1}{2}t^2}| = 1.
\end{equation*}
\end{exam}

\begin{lmm}\label{lmm3.11}
Let $A =$
$\begin{pmatrix}
  a & b\\
  c & d
\end{pmatrix}$, where $(b,n) = 1$ and $|a+d| = 2$. Then we have the following statements:

(1)If $a+d = 2$ and $\Lambda_n$ is the intertwiner obtained in Theorem \ref{thm2.3} such that $|\det(\Lambda_n)| = 1$, we have $|{\rm{Trace}}\Lambda_n| = \sqrt{n}$.

(2)If $a+d = -2$ and $\Lambda_n$ is the intertwiner obtained in Theorem \ref{thm2.4} such that $|\det(\Lambda_n)| = 1$, we have $|{\rm{Trace}}\Lambda_n| = \sqrt{n}$.
\end{lmm}
\begin{proof}
We only prove the statement (1) (the proof for the statement (2) is similar). Let
$[\gamma]\in \mathcal{X}_{SL(2,\mathbb{C})}(T^2)$, with
$\gamma(\alpha)$=
$\begin{pmatrix}
  \lambda_{1} & 0\\
  0 & \lambda_{1}^{-1}
\end{pmatrix}$ and
$\gamma(\beta)$=
$\begin{pmatrix}
  \lambda_{2} & 0\\
  0 & \lambda_{2}^{-1}
\end{pmatrix}$, be any $A$-invariant  character.

We use the same notation as in Remark \ref{rem3.2}. Then we have
\begin{equation*}
\begin{split}
s_1(1-d) + s_2b &= r_1(a-1)(1-d) + r_2b(1-d) + k_1(1-d) + r_1cb + r_2(d-1)b + k_2b\\
&=  k_1(1-d) + k_2b.\\
\end{split}
\end{equation*}
From equation \ref{eqq3.10}, we can get
\begin{equation*}
\begin{split}
&2\pi((1-d)k_1 + k_2b) = (1-d)2\pi k_1 + b2\pi k_2\\
= &(1-d)(a-1)\theta_1 + b(1-d)\theta_2 + bc\theta_1 + b(d-1)\theta_2 = 0.\\
\end{split}
\end{equation*}
Thus we have
$$(1-d)k_1 + k_2b = s_1(1-d) + s_2b = 0.$$

From equation \ref{eqq3.13}, we know
\begin{equation*}
\begin{split}
{\rm{Trace}}\Lambda_n &= n^{-\frac{1}{2}}\sum_{0\leq t\leq n-1} ((-1)^{(a+d-2)r} )^t q^{\frac{r}{2}((a+d-2)t^2 + 2(s_1(1-d) + s_2b)t)}
= n^{-\frac{1}{2}} n = \sqrt{n}.
\end{split}
\end{equation*}

\end{proof}

The following Theorem shows the limit superior related to the trace of intertwiners for any diffeomorphism of the closed torus is zero, which equals  the simplicial volume of the corresponding mapping torus.

\begin{thm}\label{thm4.1}
Let $A =\begin{pmatrix}
  a & b\\
  c & d
\end{pmatrix}$ be any fixed element in the mapping class group for the closed torus, and let $[\gamma]$ be any fixed $A$-invariant character with
$\gamma(\alpha)$=
$\begin{pmatrix}
  \lambda_{1} & 0\\
  0 & \lambda_{1}^{-1}
\end{pmatrix}$ and
$\gamma(\beta)$=
$\begin{pmatrix}
  \lambda_{2} & 0\\
  0 & \lambda_{2}^{-1}
\end{pmatrix}$. If  $1=\lambda_1^{a-1}\lambda_2^{b}$ and $1=\lambda_1^{c}\lambda_2^{d-1}$, let
 $\{\Lambda_n\}_{n\in 2\mathbb{Z}_{\geq 0}+1}$  be intertwiners obtained in Theorem \ref{thm2.3} such that $|\det(\Lambda_n)| = 1$ for all $n\in 2\mathbb{Z}_{\geq 0}+1$.
 If $1=\lambda_1^{a+1}\lambda_2^{b}$ and $1=\lambda_1^{c}\lambda_2^{d+1}$, let
 $\{\Lambda_n\}_{n\in 2\mathbb{Z}_{\geq 0}+1}$ be intertwiners obtained in Theorem \ref{thm2.4} such that $|\det(\Lambda_n)| = 1$ for all $n\in 2\mathbb{Z}_{\geq 0}+1$.
 Then we have
\begin{equation*}
\limsup_{\text{odd }n\rightarrow \infty} \frac{\log(|{\rm{Trace}}\Lambda_n|)}{n} = 0.
\end{equation*}
\end{thm}
\begin{proof}
 Since $[\gamma]$ is $A$-invariant, we have $1=\lambda_1^{a-1}\lambda_2^{b}, 1=\lambda_1^{c}\lambda_2^{d-1}$
or $1=\lambda_1^{a+1}\lambda_2^{b}, 1=\lambda_1^{c}\lambda_2^{d+1}$. We look at the case when $1=\lambda_1^{a-1}\lambda_2^{b}, 1=\lambda_1^{c}\lambda_2^{d-1}$.
Then we can set $\Lambda_n $ to be $a_n\overline{\Lambda}_n$ where $a_n = |\det(\overline{\Lambda}_n)|^{-\frac{1}{n}}$.

Case $\uppercase\expandafter{\romannumeral1}$
 when $b = 0$. In this case we know $|\det(\overline{\Lambda}_n)| = 1$ since $n^{'} = 1$.

  We have $A =\begin{pmatrix}
  1 & 0\\
  c & 1
\end{pmatrix}$  or $\begin{pmatrix}
  -1 & 0\\
  c & -1
\end{pmatrix}$. 
We first consider the case when
$A =\begin{pmatrix}
  1 & 0\\
  c & 1
\end{pmatrix}$. If $c = 0$, it is trivial. So suppose $c\neq 0$.
 Since we have $\lambda_1^{a-1}\lambda_2^{b} = 1, \lambda_1^{c}\lambda_2^{d-1} =1$, then we get $\lambda_1^{c} = 1$.
Suppose $\lambda_1 = e^{i\theta}$, then we get $\theta c = 2k\pi$ where $k$ is an integer. Since $u^n = -\lambda_1$,
we can choose $u = -e^{\frac{i\theta}{n}}q^{r}$ where $r$ is an integer. Then we have
$$u^c = (-1)^c e^{\frac{i\theta c}{n}}q^{cr} = (-1)^c e^{\frac{2k\pi i}{n} } q^{cr} = (-1)^c q^{k + cr}.$$
Note that
$|\rm{Trace}\Lambda_n|$ is
independent of the choice of $r$.

From Remark \ref{rem4.1} and Lemma \ref{lmm4.9}, we can get $\Lambda_n$ is a diagonal matrix, and $$(\overline{\Lambda}_n)_{t,t} =(v^{(d-1)}u^{c})^t q^{c(t^2-dt^2/2)}.$$
Then we have
$${\rm{Trace}}\Lambda_n = \sum_{0\leq t\leq n-1} v^{t(d-1)}u^{tc}q^{cdt^2/2} = \sum_{0\leq t\leq n-1} (-1)^{ct}q^{ct^2/2}q^{(k+cr)t}.$$
Let $\{n_i\}_{i\in\mathbb{N}}$ be a subsequence of $2\mathbb{Z}_{\geq 0}+1$
such that $(n_i, c) = 1$ for all $i$. Then for every $i$ there exists $r$ such that $k+cr \equiv 0$ (\mod $n_i$), thus
$$|{\rm{Trace}}\Lambda_{n_i}|  = |\sum_{0\leq t\leq n_i-1} (-1)^{ct}q^{ct^2/2}| = \sqrt{(n_i,c)n_i} = \sqrt{n_i}\geq 1.$$
Thus we have $$0\leq \limsup_{\text{odd }n\rightarrow \infty} \frac{\log(|{\rm{Trace}}\Lambda_n|)}{n}.$$
According to Theorem \ref{thm2.5}, we also have
$$\limsup_{\text{odd }n\rightarrow \infty} \frac{\log(|{\rm{Trace}}\Lambda_n|)}{n} \leq \limsup_{\text{odd }n\rightarrow \infty} \frac{\log(n^{\frac{3}{2}})}{n} =0.$$

We look at the case when
$A =\begin{pmatrix}
  -1 & 0\\
  c & -1
\end{pmatrix}$. Then we get $(\lambda_1)^2 = 1$ and $\lambda_1 = \pm 1$. We can choose $u = \pm 1$. From Remark \ref{rem4.1} and Lemma \ref{lmm4.9} we can get  $$(\overline{\Lambda}_n)_{t,k} = r_{k-td} (v^{(d-1)}u^{c})^t q^{c(tk-dt^2/2)},$$
where $r_k = 1$ if $k$ is a multiple of $n$ and it is zero otherwise. Then $(\overline{\Lambda}_n)_{t,t}\neq 0$ if and only if $r_{2t} \neq 0$
if and only if $n|(2t)$, which means there is only one nonzero diagonal element.
Then we get $|{\rm{Trace}}\Lambda_n| = 1$ for any $n$, which proves this special case.

Case $\uppercase\expandafter{\romannumeral2}$ when  $b\neq 0$.

Subcase (1) when $a+d\neq 2$.
 From the above discussion we know $\lambda_1,\lambda_2$ are both roots of unity, thus we can suppose $\lambda_1 = e^{i\theta_1},
 \lambda_2 = e^{i\theta_2}$ and  we can get equation \ref{eqq3.10}
where $\theta_1,\theta_2,k_1,k_2$ are determined by $\gamma$.
Since $u^n = -\lambda_1, v^n=-\lambda_2$, we can write $u =- e^{\frac{i\theta_1}{n}}q^{r_1}, v =- e^{\frac{i\theta_2}{n}}q^{r_1}$. Then we have equation \ref{eqq3.11}. Note that $|{\rm{Trace}}\Lambda_n|$ is
independent of the choice of $r_1,r_2$. 

Since  $b\neq 0$ and $2-(a+d)\neq 0$,
let $\{n_j\}_{j\in\mathbb{N}}$ be a subsequence of $2\mathbb{Z}_{\geq 0}+1$ such that $(n_j, b) = (n_j, 2-(a+d)) = 1$.

Since
$\begin{bmatrix}
  a-1 & b\\
  c & d-1
\end{bmatrix} = 2 - (a+d)$ and $(n_j, 2 - (a+d)) = 1$, the following equations always have solutions in $\mathbb{Z}_{n_j}$
\begin{equation*}
\begin{cases}
r_1(a-1) + r_2b + k_1 = 0 \\
r_1c + r_2(d-1) + k_2 = 0
\end{cases}.
\end{equation*}
Thus for every $j$, there always exist integers $r_1,r_2$ such that $s_1 = s_2 = 0$ in $\mathbb{Z}_{n_j}$. Then from equation \ref{eqq3.13}, we have
\begin{equation*}
\begin{split}
|{\rm{Trace}}\Lambda_{n_j}| &= n_{j}^{-\frac{1}{2}}|\sum_{0\leq t\leq n_j-1} ((-1)^{(a+d-2)r} )^t q^{\frac{r}{2}((a+d-2)t^2 + 2(s_1(1-d) + s_2b)t)}|\\
&= n_{j}^{-\frac{1}{2}}|\sum_{0\leq t\leq n_j-1} ((-1)^{(a+d-2)r} )^t q^{\frac{r}{2}((a+d-2)t^2 )}|\\
&= n_{j}^{-\frac{1}{2}}|\sum_{0\leq t\leq n_j-1}  (-q^{\frac{1}{2}})^{r(a+d-2)t^2 }|\\
&= n_{j}^{-\frac{1}{2}} \sqrt{(r(a+d-2),n_j)n_j} = 1.
\end{split}
\end{equation*}
Then we have $$0\leq \limsup_{\text{odd }n\rightarrow \infty} \frac{\log(|{\rm{Trace}}\Lambda_n|)}{n}.$$
From Theorem \ref{thm2.5}, we  have
$$\limsup_{\text{odd }n\rightarrow \infty} \frac{\log(|{\rm{Trace}}\Lambda_n|)}{n} = 0.$$


Subcase (2) when  $a+d = 2$. Since  $b\neq 0$,
let $\{n_k\}_{k\in\mathbb{N}}$ be a subsequence of $2\mathbb{Z}_{\geq 0}+1$ such that $(n_k, b) = 1$ for all $k$. From Lemma \ref{lmm3.11},  we
have $|{\rm{Trace}}\Lambda_{n_k}| = \sqrt{n_k}\geq 1$. Similarly we get
$$\limsup_{\text{odd }n\rightarrow \infty} \frac{\log(|{\rm{Trace}}\Lambda_n|)}{n} = 0.$$

\end{proof}

From now on we discuss the periodic mapping class. Recall that $A =\begin{pmatrix}
  a & b\\
  c & d
\end{pmatrix}$ is periodic if and only if $|a+d| \in\{ 0, 1\}$.
Suppose $[\gamma]$ is an $A$-invariant character with
 $\gamma(\alpha)$=
$\begin{pmatrix}
  \lambda_{1} & 0\\
  0 & \lambda_{1}^{-1}
\end{pmatrix}$ and
$\gamma(\beta)$=
$\begin{pmatrix}
  \lambda_{2} & 0\\
  0 & \lambda_{2}^{-1}
\end{pmatrix}$. Then we have $1=\lambda_1^{a-1}\lambda_2^{b}, 1=\lambda_1^{c}\lambda_2^{d-1}$ or
$1=\lambda_1^{a+1}\lambda_2^{b}, 1=\lambda_1^{c}\lambda_2^{d+1}$. For the case when $1=\lambda_1^{a-1}\lambda_2^{b}, 1=\lambda_1^{c}\lambda_2^{d-1}$,
 the above discussion implies $1 = \lambda_1^{2-(a+d)} = \lambda_2^{2-(a+d)}$. So if $a+d = 0$, we have $\lambda_1^{2} = \lambda_2^{2} = 1$, and then
$\lambda_1 = \pm 1, \lambda_2 = \pm 1$. Thus there is no $A$-invariant character living in the Azumaya locus if $\lambda_1 = \pm 1, \lambda_2 = \pm 1$. But we can still get  intertwiners in Theorems \ref{thm2.3}
and \ref{thm2.4} although $\lambda_1 = \pm 1, \lambda_2 = \pm 1$. Now we consider  intertwiners if we choose $\lambda_1 = \lambda_2 = 1.$

\begin{thm}
Let $A$ be a periodic mapping class, and let $\Lambda_n$ be the intertwiner obtained in Theorem \ref{thm2.3} or \ref{thm2.4} by using the trivial $A$-invariant character, that is,
$\lambda_1 = \lambda_2 = 1$, and we require $|\det(\Lambda_n)| = 1$. We have the following conclusions:

(1) If $a+d = 1$ and $\Lambda_n$ is obtained in Theorem \ref{thm2.3}, we have $|{\rm{Trace}}(\Lambda_n)| = 1$ for any odd $n$.

(2) If $a+d = -1$ and $\Lambda_n$ is obtained in Theorem \ref{thm2.4}, we have $|{\rm{Trace}}(\Lambda_n)| = 1$ for any odd $n$.

(3) If $a+d = 0$ and $\Lambda_n$ is obtained in Theorem \ref{thm2.3} or \ref{thm2.4}, we have $|{\rm{Trace}}(\Lambda_n)| = 1$ for any odd $n$.
\end{thm}

\begin{proof}
Suppose $A =\begin{pmatrix}
  a & b\\
  c & d
\end{pmatrix}$, $(b,n) = m$, $br + sn = m$ and $n = m n^{'}$. Since $\lambda_1 = \lambda_2 = 1$, we can set $ u = v = -1$. From the previous discussion we know
$$(\overline{\Lambda}_n)_{k,t} = (-1)^{cdt} r_{k-td}  q^{c(tk-dt^2/2)}$$ for all $0\leq k,t\leq n-1,$ where
$$r_{tm} =(-1)^{abrt} q^{abt^2r^2/2}, \forall 0\leq t\leq n^{'}-1,$$ and all other $r_k$ are  0.
For the $l$-th column, we have $\{(\overline{\Lambda}_n)_{ld+km, l}\}_{0\leq k\leq n^{'} - 1}$  are the only nonzero entries.
Then the $l$-th column  contains a nonzero diagonal entry if and only if $ld + km \equiv l$ (mod $n$) for some $0\leq k\leq n^{'}-1$.
It is easy to show $ld + km \equiv l$ (mod $n$) for some $0\leq k\leq n^{'}-1$ if and only if $m|(ld-l)$.

Now we suppose $(m, d-1)=1$. Then the $l$-th column of $\overline{\Lambda}_n$  contains a nonzero diagonal entry   if and only if $m|l$. Thus $(\overline{\Lambda}_n)_{tm,tm},0\leq t\leq n^{'}-1$, are  the only nonzero diagonal entries.
\begin{equation*}
\begin{split}
(\overline{\Lambda}_n)_{tm,tm} &= (-1)^{cdtm} r_{tm-tmd}  q^{c(t^2m^2-dt^2m^2/2)}\\
&= (-1)^{cdtm} (-1)^{abr(t-td)} q^{abr^2(1-d)^2t^2/2}  q^{c(t^2m^2-dt^2m^2/2)}.
\end{split}
\end{equation*}
After a similar calculation as in Remark \ref{rem3.2}, we can get
$$(\overline{\Lambda}_n)_{tm,tm} = (-1)^{(ar+dr-2r)t}(q^{\frac{m}{2}})^{(ar+dr-2r)t^2}.$$
Then we have
$${\rm{Trace}}\overline{\Lambda}_n = \sum_{0\leq t\leq n^{'}-1} (-1)^{(ar+dr-2r)t}(q^{\frac{m}{2}})^{(ar+dr-2r)t^2}.$$

Similarly if $(m, d+1) = 1$, we have
$${\rm{Trace}}\widetilde{\Lambda}_n = \sum_{0\leq t\leq n^{'}-1} (-1)^{(ar+dr+2r)t}(q^{\frac{m}{2}})^{(ar+dr+2r)t^2}.$$

(1) Since the intertwiner is obtained in Theorem \ref{thm2.3}, we can set $\Lambda_n = (n^{'})^{-\frac{1}{2}}\overline{\Lambda}_n$. We have
$d-1 = -a$ because $a+d = 1$. Then we get $(d-1, m) = 1$ because $(a,b) = 1$ and $m$ is a divisor of $b$. Then from the above discussion, we get
\begin{equation*}
\begin{split}
|{\rm{Trace}}\Lambda_n| &= (n^{'})^{-\frac{1}{2}} |\sum_{0\leq t\leq n^{'}-1} (-1)^{(ar+dr-2r)t}(q^{\frac{m}{2}})^{(ar+dr-2r)t^2}|\\
&= (n^{'})^{-\frac{1}{2}} |\sum_{0\leq t\leq n^{'}-1} (-1)^{-rt}(q^{\frac{m}{2}})^{-rt^2}|
= (n^{'})^{-\frac{1}{2}} \sqrt{(-r,n^{'})n^{'}} = 1.
\end{split}
\end{equation*}

(2) The proof is similar with (1).

(3) First we show $(m, d-1) = (m, d+1) = 1$ if $a+d = 0$. From $ad - bc = 1$, we get $-bc = d^2 + 1$.
Suppose $p|m$ and $p|d-1$, then $p|(d^2 + 1)$ and $p|(d^2-1)$. Thus we get $p|2$, which means $p=1$ because $(m,2) = 1$.
 Similarly we can show $(m, d+1) =1$. If $\Lambda_n =(n^{'})^{-\frac{1}{2}}\overline{\Lambda}_n$, then we
 \begin{equation*}
\begin{split}
|{\rm{Trace}}\Lambda_n| &= (n^{'})^{-\frac{1}{2}} |\sum_{0\leq t\leq n^{'}-1} (-1)^{(ar+dr-2r)t}(q^{\frac{m}{2}})^{(ar+dr-2r)t^2}|\\
&= (n^{'})^{-\frac{1}{2}} |\sum_{0\leq t\leq n^{'}-1} (q^m)^{-rt^2}|
= (n^{'})^{-\frac{1}{2}} \sqrt{(-r,n^{'})n^{'}} = 1.
\end{split}
\end{equation*}

If $\Lambda_n =(n^{'})^{-\frac{1}{2}}\widetilde{\Lambda}_n$, similarly we can show $|{\rm{Trace}}\Lambda_n|=1$.
\end{proof}

\section{The volume conjecture for surface diffeomorphisms: periodic case}\label{sectiiii}


\subsection{Preliminaries for the volume conjecture for periodic surface diffeomorphisms}
If  we want to formulate the parallel conjecture for periodic diffeomorphisms as in \cite{BW5,BW6}, we have to find a good invariant character that lives in the smooth part of $\mathcal{X}_{SL(2,\mathbb{C})}(S)$.

\begin{lmm}[\cite{MC}]\label{lmm4.12}
Let $A, B \in SL(2,\mathbb{C})$. If ${\rm{Trace}}([A, B]) = 2$ where $[A,B] = ABA^{-1}B^{-1}$, then $G = \langle A, B\rangle \leq SL(2,\mathbb{C})$ is  not free of rank two
where $\langle A, B\rangle$ is the group generated by $A,B$.
\end{lmm}

\begin{lmm}\label{lmm3.7}
Let $G$ be a  subgroup of $SL(2,\mathbb{C})$ freely generated by two elements, let $R$ be the subalgebra of $Mat(2,\mathbb{C})$ generated by $G$,  where $Mat(2,\mathbb{C})$ is  the algebra of all 2 by 2 complex matrices. Then
$R = Mat(2,\mathbb{C})$.
\end{lmm}

\begin{proof}
Suppose $G$ is freely generated by $A,B$. We know there exists $X\in GL(2,\mathbb{C})$ such that
$XAX^{-1}=$
$\begin{pmatrix}
  u & v\\
  0 & u^{-1}
\end{pmatrix}$,
$XBX^{-1}$=
$\begin{pmatrix}
  a & b\\
  c & d
\end{pmatrix}$.
Then $XGX^{-1}$ is a free subgroup generated by $XAX^{-1}$ and $XBX^{-1}$, and
$XRX^{-1}$ is the subalgebra generated by $XGX^{-1}$. Since $XRX^{-1} = Mat(2,\mathbb{C})$ if and only if $R = Mat(2,\mathbb{C})$,
 we can assume
 $A=$
$\begin{pmatrix}
  u & v\\
  0 & u^{-1}
\end{pmatrix}$,
$B$=
$\begin{pmatrix}
  a & b\\
  c & d
\end{pmatrix}$.

$\uppercase\expandafter{\romannumeral 1}. $  Suppose $v=0$. Then
$A=$
$\begin{pmatrix}
  u & 0\\
  0 & u^{-1}
\end{pmatrix}$ and $u^2\neq 1$, otherwise we have
$A=$
$\begin{pmatrix}
  1 & 0\\
  0 & 1
\end{pmatrix}$ or
$A=$
$\begin{pmatrix}
  -1 & 0\\
  0 & -1
\end{pmatrix}$, which contradicts the fact that $G$ is freely generated by $A,B$.
We also can get $b\neq 0$ and $c\neq 0$.
Otherwise we have ${\rm{Trace}}([A,B]) = 2$, which contradicts the fact that $A,B$ freely generate $G$ by Lemma \ref{lmm4.12}.
Since $\begin{pmatrix}
  u & 0\\
  0 & u^{-1}
\end{pmatrix},
\begin{pmatrix}
  1 & 0\\
  0 & 1
\end{pmatrix}\in R$ and $u\neq\pm1$, we have
$\begin{pmatrix}
  1 & 0\\
  0 & 0
\end{pmatrix},
\begin{pmatrix}
  0 & 0\\
  0 & 1
\end{pmatrix}\in R$. Then
$\begin{pmatrix}
  0 & b\\
  c & 0
\end{pmatrix}\in R$.
From multiplication, we get
$
\begin{pmatrix}
  0 & b\\
  0 & 0
\end{pmatrix},
\begin{pmatrix}
  0 & 0\\
  c & 0
\end{pmatrix}
\in R$, which implies $R = Mat(2,\mathbb{C})$ since $b\neq 0, c\neq 0$.

$\uppercase\expandafter{\romannumeral 2}. $ Suppose $v\neq0$. In this case we should have $c\neq 0$, otherwise
we have ${\rm{Trace}}([A,B]) = 2$, which is a contradiction.

If $u=\pm1$, then
$A$=
$\begin{pmatrix}
  u & v\\
  0 & u
\end{pmatrix}\in R$. Remember we also have
$\begin{pmatrix}
  1 & 0\\
  0 & 1
\end{pmatrix}\in R$, which implies
$\begin{pmatrix}
  0 & v\\
  0 & 0
\end{pmatrix}\in R$. Furthermore we have
$\begin{pmatrix}
  0 & 1\\
  0 & 0
\end{pmatrix}\in R$ because $v\neq 0$.
From $
\begin{pmatrix}
  1 & 0\\
  0 & 1
\end{pmatrix},
\begin{pmatrix}
  0 & 1\\
  0 & 0
\end{pmatrix},
\begin{pmatrix}
  a & b\\
  c & d
\end{pmatrix}
\in R$, we can get
$
\begin{pmatrix}
  a & 0\\
  c & d
\end{pmatrix}
\in R$, and also
$
\begin{pmatrix}
  a-d & 0\\
  c & 0
\end{pmatrix}
\in R$.
From multiplication, we get
$
\begin{pmatrix}
  0 & 1\\
  0 & 0
\end{pmatrix}
\begin{pmatrix}
  a-d & 0\\
  c & 0
\end{pmatrix}=
\begin{pmatrix}
  c & 0\\
  0 & 0
\end{pmatrix}
\in R$, which implies
$\begin{pmatrix}
  1 & 0\\
  0 & 0
\end{pmatrix}
\in R$ because $c\neq 0$.
We  have
$\begin{pmatrix}
  1 & 0\\
  0 & 1
\end{pmatrix},
\begin{pmatrix}
  0 & 1\\
  0 & 0
\end{pmatrix},
\begin{pmatrix}
  1 & 0\\
  0 & 0
\end{pmatrix}\in R$, then
$\begin{pmatrix}
  0 & 0\\
  0 & 1
\end{pmatrix},
\begin{pmatrix}
  0 & 1\\
  0 & 0
\end{pmatrix},
\begin{pmatrix}
  1 & 0\\
  0 & 0
\end{pmatrix}\in R$.
Remember we also have $
\begin{pmatrix}
  a & b\\
  c & d
\end{pmatrix}
\in R$ and $c\neq 0$, which implies $R= Mat(2,\mathbb{C})$.

If $u\neq \pm1$. We have
$\begin{pmatrix}
  u-u^{-1} & v\\
  0 & 0
\end{pmatrix},
\begin{pmatrix}
  0 & v\\
  0 & u^{-1}-u
\end{pmatrix}
\in R$, so
$\begin{pmatrix}
  1 & k\\
  0 & 0
\end{pmatrix},
\begin{pmatrix}
  0 & -k\\
  0 & 1
\end{pmatrix}
\in R$ where $k=v/(u-u^{-1})$.
Then from multiplication, we get
$$\begin{pmatrix}
  1 & k\\
  0 & 0
\end{pmatrix}
\begin{pmatrix}
  a & b\\
  c & d
\end{pmatrix}=
\begin{pmatrix}
  a+kc & b+kd\\
  0 & 0
\end{pmatrix}
\in R.$$
Next we want to show $b+kd\neq k(a+kc)$. Suppose on the contrary. Then we have $b+kd = k(a+kc) = ka + k^2 c$.
With $k= v/(u-u^{-1})$, we can get
\begin{align*}
b+\frac{dv}{u-u^{-1}} = \frac{av}{u-u^{-1}} + \frac{cv^2}{(u-u^{-1})^2}
\Longrightarrow &2b=  bu^2+bu^{-2} + dvu-dvu^{-1} - avu + avu^{-1} - cv^2.
\end{align*}
Then we get
\begin{align*}
&{\rm{Trace}}([A,B])= ad+acuv+cdvu^{-1}+c^2v^2 - cbu^2 - cduv -cbu^{-2} - cau^{-1}v + ad\\
= & 2ad -c(-auv-dvu^{-1}-cv^2 + bu^2 + duv + bu^{-2} + au^{-1}v)
=  2ad - 2cb
=  2.
\end{align*}
Since ${\rm{Trace}}([A,B]) = 2$ is a contradiction, we have $b+kd\neq k(a+kc)$.
We can get $\begin{pmatrix}
  1 & 0\\
  0 & 0
\end{pmatrix},
\begin{pmatrix}
  0 & 1\\
  0 & 0
\end{pmatrix}
\in R$ because
$\begin{pmatrix}
  a+kc & b+kd\\
  0 & 0
\end{pmatrix},
\begin{pmatrix}
  1 & k\\
  0 & 0
\end{pmatrix}
\in R$.
We also have
$\begin{pmatrix}
  1 & 0\\
  0 & 1
\end{pmatrix},
\begin{pmatrix}
  a & b\\
  c & d
\end{pmatrix}
\in R$, then we can get
$\begin{pmatrix}
  1 & 0\\
  0 & 0
\end{pmatrix},
\begin{pmatrix}
  0 & 1\\
  0 & 0
\end{pmatrix}
\begin{pmatrix}
  0 & 0\\
  0 & 1
\end{pmatrix},
\begin{pmatrix}
  a & b\\
  c & d
\end{pmatrix}
\in R$.
So $R = Mat(2,\mathbb{C})$ because $c\neq 0$.

\end{proof}

\begin{prop}\label{prop3.1}
Let $\gamma: \pi_{1}(S)\rightarrow SL(2,\mathbb{C})$ be a representative of an element in the character variety $\mathcal{X}_{SL(2,\mathbb{C})}(S)$.
Then $\gamma$ is irreducible if ${\rm{Im}}\gamma$ contains a subgroup of $SL(2,\mathbb{C})$ freely generated by two elements.
In particular $\gamma$ is irreducible if $S$ has negative Euler characteristic and $\gamma$ is injective.
\end{prop}
\begin{proof}
Lemma \ref{lmm3.7}.

\end{proof}

\subsection{Statement of the  conjecture }\label{sub5.2}


To get the intertwiner, we first have to get a $\varphi$-invariant  smooth character $\gamma\in \mathcal{X}_{SL(2,\mathbb{C})}(S)$.
At page 371 of \cite{BM} it is proved
that every periodic
 diffeomorphism fixes a point in  the Teichm{\"u}ller space.
This means there is a discrete and faithful group homomorphism
$\bar{\gamma} :\pi_1(S) \rightarrow PSL(2,\mathbb{R})$ such that $\bar{\gamma}\varphi_{*}$ is conjugate to
$\bar{\gamma}$ by  an element in $PSL(2,\mathbb{R})$  where $\varphi_*$ is the
isomomorphism from $\pi_1(S)$ to $\pi_1(S)$ induced by $\varphi$.

Since $PSL(2,\mathbb{R})\subset PSL(2,\mathbb{C})$, we can regard $\bar{\gamma}\varphi_{*}$ and
$\bar{\gamma}$ as two elements in $\mathcal{X}_{PSL(2,\mathbb{C})}(S)$. Then
$\bar{\gamma}\varphi_{*}$ is conjugate to
$\bar{\gamma}$ by  an element in $PSL(2,\mathbb{C})$. Thus $\bar{\gamma}$ can be extended to
a group homomorphism from $\pi_1(M_{\varphi})$ to $PSL(2,\mathbb{C})$, we use  $\hat{\gamma}$ to denote
this homomorphism. Then we can lift $\hat{\gamma}$ to a group homomorphism $\tilde{\gamma}$ from $\pi_1(M_{\varphi})$ to $SL(2,\mathbb{C})$. 
The restriction of $\tilde{\gamma}$ to $\pi_1(S)$ is $\varphi$-invariant, and we use $\gamma$ to denote this group homomorphism.
Note that $\gamma$ is a group homomorphism from $\pi_1(S)$ to $SL(2,\mathbb{C})$. Let $\varepsilon$ be the projection from
$SL(2,\mathbb{C})$ be $PSL(2,\mathbb{C})$, then we have $\varepsilon\tilde{\gamma} = \hat{\gamma}$, furthermore we have
$$\varepsilon\gamma = \varepsilon\tilde{\gamma}|_{\pi_1(S)} = \hat{\gamma}|_{\pi_1(S)} = \bar{\gamma}.$$

Since $\bar{\gamma}$ is injective, we have $\gamma$ is injective. From Proposition \ref{prop3.1}, we know $\gamma$ is irreducible. Thus we get a
$\varphi$-invariant  smooth character $\gamma\in \mathcal{X}_{SL(2,\mathbb{C})}(S)$. From now on, we use $\gamma_{\varphi}$ to denote
$\gamma$ and $\overline{\gamma_{\varphi}}$ to denote $\overline{\gamma}$.

For every puncture $v$ in $S$, we know ${\rm{Trace}}\gamma_{\varphi}(\alpha_v) = \pm2$ where $\alpha_v$ is the loop  going around puncture $v$.
If ${\rm{Trace}}\gamma_{\varphi}(\alpha_v) = 2$, we choose  $p_v = -(q + q^{-1})$. Then 
$$T_n(p_v) = (-q)^n + (-q^{-1})^n = -1 -1 = -{\rm{Trace}}\gamma_{\varphi}(\alpha_v).$$
If ${\rm{Trace}}\gamma_{\varphi}(\alpha_v) = -2$, we choose $p_v = 1 + 1$. Then
$$T_n(p_v) = 1^n + 1^n = 1 + 1 = -{\rm{Trace}}\gamma_{\varphi}(\alpha_v).$$
Since ${\rm{Trace}}\gamma_{\varphi}(\alpha_v) = {\rm{Trace}}\gamma_{\varphi}(\varphi(\alpha_v)) = {\rm{Trace}}\gamma_{\varphi}(\alpha_{\varphi(v)})$, we have $p_v = p_{\varphi(v)}$.
So now we have everything we want. Then we obtain the Kauffman bracket intertwiner $\Lambda_{\varphi,\gamma_{\varphi}}^{q}$ associated to these data. We require
$|\det(\Lambda_{\varphi,\gamma_{\varphi}}^{q})| = 1$. With the fixed  $S$, $\varphi$,  $\gamma_{\varphi}$ and $\{p_v\}_{v}$, we have $|{\rm{Trace}} \Lambda^{q}_{\varphi,r}|$
is only related to $q$. 
\begin{conj}\label{conj4.2}
Suppose $S$ is an oriented surface with negative Euler characteristic, and
 $\varphi$ is a periodic diffeomorphism for $S$.  Let $\gamma_{\varphi}$  be the $\varphi$-invariant smooth character defined as in the second paragraph of this subsection.  For each puncture $v$, let $p_v$ be the complex number defined as in the third paragraph of this subsection. Let $q_n = e^{2\pi i/n}$ with $(q_n)^{1/2} = e^{\pi i/n}$. Then
\begin{equation*}
 \lim_{n\text{\;} odd \rightarrow \infty} \frac{1}{n} \log |{\rm{Trace}} \Lambda^{q_n}_{\varphi,r}| = 0.
 \end{equation*}
\end{conj}

\subsection{Proofs for the conjecture  for some special cases}
In the remaining part of this paper, we will present some results related to our conjecture. Especially, we will  give a proof  for our conjecture when the surface $S$ is the once punctured torus.

In the following Theorem, we use the periodic property of the diffeomorphisms to prove that the limit in Conjecture \ref{conj4.2} is less than or equal to zero if it exists.

\begin{thm}\label{th4.1}
If $\lim_{n\text{\;} odd \rightarrow \infty} \frac{1}{n} \log |{\rm{Trace}} \Lambda^{q_n}_{\varphi,r_{\varphi}}|$ exists, the limit is less than or equal to zero.
\end{thm}
\begin{proof}
Let $\rho: SK_{q_n^{1/2}}(S)\rightarrow {\rm{End}}(V)$ be an irreducible representation of the skein algebra associated to $\gamma_{\varphi}$ and weight system $\{p_v\}_{v}$.
From the definition of  intertwiners $\Lambda^{q_n}_{\varphi,r_{\varphi}}$, we know $$\rho\varphi_{\sharp}(X) = \Lambda^{q_n}_{\varphi,r_{\varphi}} \circ \rho(X) \circ ( \Lambda^{q_n}_{\varphi,r_{\varphi}})^{-1} $$ for all $X\in SK_{q_n^{1/2}}(S)$. We have
$$\rho(\varphi^2)_{\sharp}(X)= \rho\varphi_{\sharp}(\varphi_{\sharp}(X))  =
 \Lambda^{q_n}_{\varphi,r_{\varphi}} \circ \rho\varphi_{\sharp}(X) \circ ( \Lambda^{q_n}_{\varphi,r_{\varphi}})^{-1} =
 (\Lambda^{q_n}_{\varphi,r_{\varphi}})^2 \circ \rho\varphi_{\sharp}(X) \circ ( \Lambda^{q_n}_{\varphi,r_{\varphi}})^{-2} .$$
 Then it is easy to show that, with any integer $j$, we have
 $$\rho(\varphi^j)_{\sharp}(X)=
 (\Lambda^{q_n}_{\varphi,r_{\varphi}})^j \circ \rho\varphi_{\sharp}(X) \circ ( \Lambda^{q_n}_{\varphi,r_{\varphi}})^{-j} .$$

 Since $\varphi$ is periodic, there exists a positive integer $k$ such that $\varphi^k = Id_{S}$.
 Then we have
$$\rho(X)=\rho(\varphi^k)_{\sharp}(X) = (\Lambda^{q_n}_{\varphi,r_{\varphi}})^k \circ \rho(X) \circ ( \Lambda^{q_n}_{\varphi,r_{\varphi}})^{-k} $$
for all $X\in SK_{q^{1/2}}(S)$. We must have $(\Lambda^{q_n}_{\varphi,r_{\varphi}})^k = \lambda I$ because $\rho$ is irreducible, where $I$ is the identity matrix and $\lambda$
is a nonzero complex number. But we require $|\det(\Lambda^{q_n}_{\varphi,r_{\varphi}})| = 1$, thus
$|\lambda| = 1$. Actually we can always choose a good $\Lambda^{q_n}_{\varphi,r_{\varphi}}$ such that $(\Lambda^{q_n}_{\varphi,r_{\varphi}})^k = I$.
Since $x^k - 1$ has no multiple roots, then $\Lambda^{q_n}_{\varphi,r_{\varphi}}$ is always diagonalizable.  All its eigenvalues are $k$-roots of unity.
Then
$${\rm{Trace}}\Lambda^{q_n}_{\varphi,r_{\varphi}} = \sum_{0\leq i\leq n-1}\lambda_i,$$
where $\lambda_i^{k} = 1$ for all $0\leq i\leq n-1$.

We have $|{\rm{Trace}}\Lambda^{q_n}_{\varphi,r_{\varphi}}|\leq n$. So if the limit  exists, we have the limit is less than or equal to zero.
\end{proof}

From the proof of Theorem \ref{th4.1}, we know $|{\rm{Trace}}\Lambda^{q_n}_{\varphi,r_{\varphi}}|$ is simply the absolute value of the sum of roots of unity.
We are only concerned with  how small  $|{\rm{Trace}}\Lambda^{q_n}_{\varphi,r_{\varphi}}|$ can be because of Theorem \ref{th4.1}. Actually this problem was already asked by
Myerson \cite{GM} and Tao \cite{TT}. For any two positive integers $k,n$, let $f(n,k)$ be the least absolute value of a nonzero sum of $n$ (not necessarily distinct) $k$-th roots of unity. Myerson gave the lower bound for all positive integers $k,n$
\begin{equation}\label{eq5.4}
f(n,k)\geq n^{-k}.
\end{equation}
According to \cite{TK}, we know ${\rm{Trace}} \Lambda^{q_n}_{\varphi,r}\neq 0$ if the order of $\varphi$ is $2^m$ for some positive integer $m$.

\begin{thm}\label{thm5.16}
If $\varphi$ is of order $2^m$ where $m$ is any positive integer, for any surface with negative Euler characteristic, we have
$$ \lim_{n\text{\;} odd \rightarrow \infty} \frac{1}{n} \log |{\rm{Trace}} \Lambda^{q_n}_{\varphi,r_{\varphi}}| = 0.$$
\end{thm}
\begin{proof}
Since for any odd $n$, we have ${\rm{Trace}} \Lambda^{q_n}_{\varphi,r}\neq 0$. Then
$$n^{-k}\leq f(n,k)\leq|{\rm{Trace}} \Lambda^{q_n}_{\varphi,r}|,$$
where $k=2^m$.
So we get
$$\frac{1}{n} \log n^{-k} \leq\frac{1}{n} \log |{\rm{Trace}} \Lambda^{q_n}_{\varphi,r_{\varphi}}|\leq \frac{1}{n} \log n.$$
Then
$\lim_{n\text{\;} odd \rightarrow \infty} \frac{1}{n} \log |{\rm{Trace}} \Lambda^{q_n}_{\varphi,r_{\varphi}}| = 0$.
\end{proof}

\begin{prop}
If $\varphi$ is of order $p^m$ where $p$ is any positive prime number and $m$ is any positive integer, for any surface with negative Euler characteristic, we have
$$ \limsup_{n\text{\;} odd \rightarrow \infty} \frac{1}{n} \log |{\rm{Trace}} \Lambda^{q_n}_{\varphi,r_{\varphi}}| = 0.$$
\end{prop}
\begin{proof}
The proof is similar with Theorem \ref{thm5.16}.
\end{proof}

For the following discussion, we will use some notations and terminologies  in \cite{BW5}.
Suppose the surface $S$ has at least one puncture, that is, it has ideal triangulations. Let $\tau$ be an ideal triangulation of $S$, and let $\varphi$
 be any periodic map of $S$. Suppose $\tau = \tau^{(0)}, \tau^{(1)},\dots,\tau^{(k)} = \varphi(\tau)$ is an ideal triangulation sweep.
 Since $\varphi$ fixes a point in the Teichm{\"u}ller space, there exists a periodic edge weight system
 $a = a^{(0)}, a^{(1)},\dots,a^{(k)} = a\in (\mathbb{R}_{>0})^{e}$ where $a$ is the shear parameter corresponding to this fixed point in the Teichm{\"u}ller space.
Then $[\overline{\gamma_{\varphi}}]\in \mathcal{X}_{PSL(2,\mathbb{C})}(S)$ is the character associated to the weight system $a$. From the above discussion, we know
 $[\overline{\gamma_{\varphi}}]$ can be lift to a smooth $\varphi$-invariant character $[\gamma_{\varphi}]$ in $\mathcal{X}_{SL(2,\mathbb{C})}(S)$.

 We also have $a_{i_1}a_{i_2}\dots a_{i_j} = 1$, where $e_{i_1}, e_{i_2},\dots, e_{i_j}$
 are all the edges connecting to a common vertex, because $a$ corresponds to a complete hyperbolic structure.
 If ${\rm{Trace}}\gamma_{\varphi}(\alpha_v) = 2$, set  $h_v = q^2$. Then $h_v^{n} = 1$ and $p_v^{2} = h_v + h_v^{-1} + 2$. 
If ${\rm{Trace}}\gamma_{\varphi}(\alpha_v) = -2$, set $h_v = 1$. Then
$h_v^{n} = 1$ and $p_v^{2} = h_v + h_v^{-1} + 2$. Obviously $h_v = h_{\varphi(v)}$ for any puncture $v$.
Proposition 15 in \cite{BW5} implies that we can obtain an interwiner $\overline{\Lambda}^{q}_{\varphi,\overline{r_{\varphi}}}$
with $|\det(\overline{\Lambda}^{q}_{\varphi,\overline{r_{\varphi}}})| = 1$. According to Theorem 16 in \cite{BW5}, we have $|{\rm{Trace}}\overline{\Lambda}^{q}_{\varphi,\overline{r_{\varphi}}}| = |{\rm{Trace}} \Lambda^{q}_{\varphi,r_{\varphi}}|$.

For the once punctured torus $S_{1,1}$, we only have one puncture $v$. Let $\alpha = K_2, \beta = K_1$ denote two elements in $\pi_1(S_{1,1})$, see Figure \ref{fg3.2}. It is well-known that
$\alpha,\beta$ freely generate $\pi_1(S_{1,1})$. Let $\alpha_v$ be the loop around $v$. Then $\alpha_v = \beta\alpha\beta^{-1}\alpha^{-1}$.
From Lemma \ref{lmm4.12},
we have
$${\rm{Trace}} \gamma_{\varphi}(\alpha_v) = {\rm{Trace}} \gamma_{\varphi}(\beta\alpha\beta^{-1}\alpha^{-1}) =
{\rm{Trace}} \gamma_{\varphi}(\beta) \gamma_{\varphi}(\alpha) \gamma_{\varphi}(\beta)^{-1} \gamma_{\varphi}(\alpha)^{-1}\neq 2$$
because $\gamma_{\varphi}$ is injective. Thus we must have ${\rm{Trace}} \gamma_{\varphi}(\alpha_v) = -2$, which means  $h_v = 1$.

\begin{lmm}\label{lmm4.15}
Let the surface be $S_{1,1}$, then Conjecture \ref{conj4.2} holds for $\varphi$ being $\begin{pmatrix}
  0 & 1\\
  -1 & -1
\end{pmatrix}$ \\or
$\begin{pmatrix}
  1 & 1\\
  -1 & 0
\end{pmatrix}$.
\end{lmm}

\begin{proof}
We only prove the case when $\varphi = \begin{pmatrix}
  0 & 1\\
  -1 & -1
\end{pmatrix}$ (the proof for the other one is similar with this one).
Let $\tau$ be the ideal triangulation in Figure \ref{fgggg3}. Then $\varphi(\tau)$ is the following ideal triangulation in Figure \ref{fg55}.
\begin{figure}[htbp]
  \centering
  \includegraphics[scale=0.5]{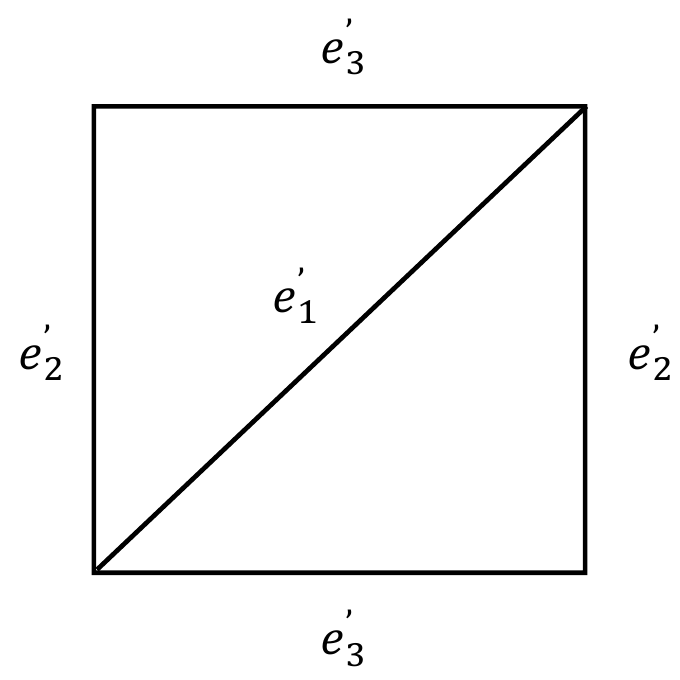}
\caption{}\label{fg55}
\end{figure}

Thus from $\tau$ to $\varphi(\tau)$ is relabeling. Suppose the shear parameter for $\tau$ is $a^{\tau} = (a_1, a_2, a_3)$, then
 the shear parameter for $\varphi(\tau)$ is $a^{\varphi(\tau)} = (a_3, a_1, a_2)$. From $a^{\tau} = a^{\varphi(\tau)}$, we get
 $a_1 = a_2 = a_3$. Since we also have $a_1^{2}a_2^{2}a_3^{2} = 1$ and $a_i\in\mathbb{R}_{>0}$, we have $a_1 = a_2 = a_3 = 1$.

  Recall that the Chekhov-Fock algebra associated to the ideal triangulation $\tau$ is
 $\mathbb{C}_{q^4}[Y_1^{\pm1}, Y_2^{\pm1}, Y_3^{\pm1}]$, where $Y_i$ corresponds to edge $e_i$ for $i=1,2,3$.
 The algebra $\mathbb{C}_{q^4}[Y_1^{\pm1}, Y_2^{\pm1}, Y_3^{\pm1}]$ is generated by $Y_1,Y_2,Y_3$ and subject to relations:
 $$Y_1Y_2 = q^4Y_2Y_1, Y_2Y_3 = q^4 Y_3Y_2, Y_3Y_1 = q^4Y_1Y_3, Y_iY_i^{-1} = Y_i^{-1}Y_i = 1.$$

 Define the irreducible representation $\rho$ of $\mathbb{C}_{q^4}[Y_1^{\pm1}, Y_2^{\pm1}, Y_3^{\pm1}]$
 as $\rho_{1,1,1}$ in Lemma \ref{lmm2.9}, that is, set $y_1 = y_2 = y_3 = 1$ in  equation \ref{eq2.17}.
 Then
 $$\rho(Y_1^{n}) = Id_V = a_1Id_V, \rho(Y_2^{n}) = Id_V = a_2Id_V, \rho(Y_3^{n}) = Id_V = a_3Id_V$$
 and
 $$\rho(H_v) = \rho([Y_1^{2}Y_2^{2}Y_3^{2}]) = Id_V = h_v Id_V.$$

It is easy to calculate that $\Phi_{\tau\varphi(\tau)}^{q_n}\Psi_{\varphi(\tau)\tau}^{q_n}$ is actually an isomorphism from $\mathbb{C}_{q^4}[Y_1^{\pm1}, Y_2^{\pm1}, Y_3^{\pm1}]$
to itself and
$$\Phi_{\tau\varphi(\tau)}^{q_n}\Psi_{\varphi(\tau)\tau}^{q_n}(Y_1) = Y_3,
\Phi_{\tau\varphi(\tau)}^{q_n}\Psi_{\varphi(\tau)\tau}^{q_n}(Y_2) = Y_1,
\Phi_{\tau\varphi(\tau)}^{q_n}\Psi_{\varphi(\tau)\tau}^{q_n}(Y_3) = Y_2.$$
We use $\rho^{'}$ to denote the irreducible representation $\rho\Phi_{\tau\varphi(\tau)}^{q_n}\Psi_{\varphi(\tau)\tau}$. Then $\rho$ is isomorphic to
$\rho^{'}$.

For each $0\leq k\leq n-1$, set $$v_k = \sum_{0\leq i\leq n-1}q_n^{k^2+i^2 + 4ik + i-k}w_i.$$ Then we have
$$
\rho^{'}(Y_1)(v_k) = q_n^{4k}v_k,
\rho^{'}(Y_2)(v_k) =  q_n^{-2k}v_{k+1},
\rho^{'}(Y_3)(v_k) = q_n^{-2k}v_{k-1}.
$$
Define invertible operator $\Lambda$ for $V$ such that $\Lambda(w_k) = v_k, \forall 0\leq k\leq n-1$.
Then, for all $0\leq k\leq n-1$, we have
$$\rho^{'}(Y_1)(\Lambda(w_k)) = \rho^{'}(Y_1)(v_k) = q_n^{4k}v_k = \Lambda(q_n^{4k}w_k) = \Lambda(\rho(Y_1)w_k).$$
Thus we get $\rho^{'}(Y_1) = \Lambda \circ\rho(Y_1)\circ \Lambda^{-1}$.
Similarly we can show $\rho^{'}(Y_2) = \Lambda \circ\rho(Y_2)\circ \Lambda^{-1}$ and $\rho^{'}(Y_3) = \Lambda \circ\rho(Y_3)\circ \Lambda^{-1}$. Thus $\Lambda$
is the intertwiner.
As a matrix, we have $\Lambda_{i,k} = q_n^{k^2+i^2 + 4ik + i-k}$.

From pure calculate, we get $|\det(\Lambda)| = n^{\frac{n}{2}}$. Thus we can set
$\overline{\Lambda}^{q_n}_{\varphi,\overline{r_{\varphi}}} = n^{-\frac{1}{2}}\Lambda$.
Then $$|{\rm{Trace}} \Lambda^{q_n}_{\varphi,r_{\varphi}}|
=|{\rm{Trace}}\overline{\Lambda}^{q_n}_{\varphi,\overline{r_{\varphi}}}| = n^{-\frac{1}{2}} \sum_{0\leq i\leq n-1} q_n^{6i^2} = n^{-\frac{1}{2}}\sqrt{(6,n)n} = \sqrt{(6,n)}.$$
Obviously we get $$ \lim_{n\text{\;} odd \rightarrow \infty} \frac{1}{n} \log |{\rm{Trace}} \Lambda^{q_n}_{\varphi,r_{\varphi}}| = 0.$$
\end{proof}

\begin{rem}
In the proof of Lemma \ref{lmm4.15}, when we try to find the periodic edge weight system for the triangulation sweep $\tau,\varphi(\tau)$, we require
$a_i\in \mathbb{R}_{>0}$ because we want to get the fixed character corresponding to a point in the Teichm{\"u}ller space. Actually we still get the same
intertwiner $\Lambda$ as in Lemma \ref{lmm4.15} without requiring $a_i\in \mathbb{R}_{>0}$, that is, for any periodic edge weight system, the intertwiner we get is
$\Lambda$. This means Lemma  \ref{lmm4.15} still holds when we choose any other $\varphi$-invariant smooth character (without the restriction for only choosing the one corresponding to a fixed point in  the Teichm{\"u}ller space). Readers can check the same arguments hold for Theorems \ref{thm5.16} and \ref{thm5.14}.
\end{rem}

Let $\phi$ be a pseudo-Anosov  map for $S$, and let $f$ be any  diffeomorphism for surface $S$. Then $f\phi f^{-1}$ is also a pseudo-Anosov  map.
Then we have the following conclusion:
\begin{lmm}\label{lmm4.16}
Let $\phi$ be any  pseudo-Anosov  map for $S$, and let $f$ be any  diffeomorphism for $S$. If Conjecture \ref{conj4.1} holds for $\phi$, then
it also holds for $f\phi f^{-1}$.
\end{lmm}

\begin{proof}
We will use the same notations as in Conjecture \ref{conj4.1}. Let $f^{-1}_{*}$ be the isomorphism from $\pi_1(S)$ to itself induced by $f^{-1}$. Then
 $[\gamma f^{-1}_{*}]$ is  a smooth $f\phi f^{-1}$-invariant character. Set $\theta_v^{'} = \theta_{f^{-1}(v)}$, then
  $\theta_v^{'}$ are invariant under the action of $f\phi f^{-1}$ and
  $${\rm{Trace}}  \gamma f^{-1}_{*}(\alpha_{v}) = {\rm{Trace}}  \gamma (\alpha_{f^{-1}(v)}) = -e^{\theta_{f^{-1}(v)}} - e^{-\theta_{f^{-1}(v)}}
  =-e^{\theta_v^{'}} - e^{-\theta_v^{'}}.$$
 Set $p_v^{'} = e^{\frac{\theta_v^{'}}{n}} + e^{-\frac{\theta_v^{'}}{n}} =
 e^{\frac{\theta_{f^{-1}(v)}}{n}} + e^{-\frac{\theta_{f^{-1}(v)}}{n}} = p_{f^{-1}(v)}$, then
 $$T_n(p_v^{'}) = -{\rm{Trace}}\gamma f^{-1}_{*}(\alpha_{v}).$$

Recall that we use $f_{\sharp}^{-1}$ to denote the isomorphism from $SK_{q^{1/2}}(S)$ to itself induced by $f^{-1}$.
 Let $\rho$ be the irreducible representation associated to $[\gamma]$ and puncture weights $p_v$. Then
 $\rho f_{\sharp}^{-1}$ is  an irreducible representation associated to character
 $[\gamma f^{-1}_{*}]$ and puncture weights $p_v^{'}$.

With the assumption for Conjecture \ref{conj4.1}, we have
 $$\rho\phi_{\sharp}(X) = \Lambda^{q_n}_{\phi,r} \circ \rho(X) \circ ( \Lambda^{q_n}_{\phi,r})^{-1} $$
 for any element $X\in SK_{q^{1/2}}(S)$ and $|\det(\Lambda^{q_n}_{\phi,r})| = 1$.
 Then we get
 \begin{equation*}
 \begin{split}
 \rho f_{\sharp}^{-1}(f\phi f^{-1})_{\sharp}(X) &= \rho f_{\sharp}^{-1}f_{\sharp}\phi_{\sharp} f^{-1}_{\sharp}(X)
 = \rho\phi_{\sharp} f^{-1}_{\sharp}(X)=\Lambda^{q_n}_{\phi,r} \circ \rho(f^{-1}_{\sharp}(X)) \circ ( \Lambda^{q_n}_{\phi,r})^{-1}.
 \end{split}
 \end{equation*}
 Thus we get $\Lambda^{q_n}_{f\phi f^{-1},rf^{-1}_{*}} = \Lambda^{q_n}_{\phi,r}$, and
 \begin{equation*}
 \begin{split}
 &\lim_{n\text{\;} odd \rightarrow \infty} \frac{1}{n} \log |{\rm{Trace}} \Lambda^{q_n}_{f\phi f^{-1},rf^{-1}_{*}}|
 = \lim_{n\text{\;} odd \rightarrow \infty} \frac{1}{n} \log |{\rm{Trace}} \Lambda^{q_n}_{\phi,r}| \\
 = &\frac{1}{4\pi}vol_{hyp}(M_\phi) = \frac{1}{4\pi}vol_{hyp}(M_{f\phi f^{-1}}).
  \end{split}
 \end{equation*}
\end{proof}

From \cite{BW6}, we know Conjecture \ref{conj4.1} holds for $\phi=\begin{pmatrix}
  2 & 1\\
  1 & 1
\end{pmatrix}$.

\begin{cor}
Conjecture \ref{conj4.1} holds for all $f\phi f^{-1}$ where $f$ is any element in $GL(2,\mathbb{Z})$.
\end{cor}

Let $\varphi$ be a periodic map for $S$, and let $g$ be any diffeomorphism for $S$. Then $g\varphi g^{-1}$ is also a periodic map.
The same discussion as in Lemma \ref{lmm4.16} implies the following conclusion.

\begin{lmm}\label{lmm4.17}
Let $\varphi$ be any  periodic  map for $S$, and let $g$ be any  diffeomorphism for $S$. If Conjecture \ref{conj4.2} holds for $\varphi$, then
it also holds for $g\varphi g^{-1}$.
\end{lmm}

The following Theorem shows Conjecture \ref{conj4.2} holds for the once punctured torus. This  confirms the relation between the intertwiner and  the simplicial volume of the corresponding mapping torus.

\begin{thm}\label{thm5.14}
Conjecture \ref{conj4.2} holds for the once punctured torus.
\end{thm}
\begin{proof}
Let $\varphi$ be any periodic map for $S_{1,1}$. Then the order of $\varphi$ could be 1,2,3,4 or 6.
According to Theorem \ref{thm5.16}, Conjecture \ref{conj4.2} holds if the order of $\varphi$ is 2 or 4.

If the order of $\varphi$ is 1, then $\varphi$ is just the identity map. In this case, we can just choose the intertwiner to be the identity operator. Then Conjecture \ref{conj4.2} holds trivially.

We look at the case when the order of $\varphi$ is 3 or 6. For these two cases, we have $|{\rm{Trace}}\varphi| = 1$.
According to \cite{KO}, we know there exists an element $g\in GL(2,\mathbb{Z})$ such that
 $\varphi = g\begin{pmatrix}
  0 & 1\\
  -1 & -1
\end{pmatrix} g^{-1}$ or
$\varphi = g\begin{pmatrix}
  1 & 1\\
  -1 & 0
\end{pmatrix}g^{-1}$. From Lemmas \ref{lmm4.15} and \ref{lmm4.17}, we get Conjecture \ref{conj4.2} holds for these two cases.

\end{proof}

\begin{rem}
From the proof of Theorem \ref{thm5.16}, we know if we can show ${\rm{Trace}} \Lambda^{q_n}_{\varphi,r}\neq 0$ after $n$
is big enough, then we can prove
$$\lim_{n\text{\;} odd \rightarrow \infty} \frac{1}{n} \log |{\rm{Trace}} \Lambda^{q_n}_{\varphi,r_{\varphi}}| = 0.$$
\end{rem}

\begin{rem}
From subsection \ref{newsec}, we know the periodic edge weight system $a = a^{(0)},a^{(1)},\dots, a^{(k)} = a$ for the ideal triangulation sweep $\tau = \tau^{(0)}, \tau^{(1)},\dots, \tau^{(k)} = \varphi(\tau)$ and $\varphi$-invariant puncture weights $h_v$ can give us
the intertwiner $\overline{\Lambda}^{q}_{\varphi,\overline{\gamma}}$ such that $$\overline{\rho}\circ \Phi^{q}_{\tau\varphi(\tau)}\circ \Psi_{\varphi(\tau)\tau}^{q} (X) =
  \overline{\Lambda}^{q}_{\varphi,\overline{\gamma}} \circ \overline{\rho}(X) \circ (\overline{\Lambda}^{q}_{\varphi,\overline{\gamma}})^{-1}$$
  for every $X\in \mathcal{T}^{q}_{\tau}$.

  It is easy to verify that $(\Phi^{q}_{\tau\varphi(\tau)}\circ \Psi_{\varphi(\tau)\tau}^{q})^m =
  (\Phi^{q}_{\tau\varphi^m(\tau)}\circ \Psi_{\varphi^m(\tau)\tau}^{q})$, and
   $$a = a^{(0)},a^{(1)},\dots, a^{(k)},\dots\dots, a^{(0)},a^{(1)},\dots, a^{(k)} = a$$ is the periodic edge weight system for the ideal triangulation sweep
  \begin{align*}
  \tau = \tau^{(0)}, \tau^{(1)},\dots, \tau^{(k)} = \varphi(\tau), \varphi(\tau^{(1)}),\dots, \varphi(\tau^{(k)}) = \varphi^2(\tau),\dots,\\ \varphi^{m-1}(\tau) = \varphi^{m-1}(\tau^{(0)}), \varphi^{m-1}(\tau^{(1)}),\dots, \varphi^{m-1}(\tau^{(k)}) = \varphi^m(\tau)
  \end{align*}
  and $h_{\varphi^{m}(v)} = h_v$.

  Suppose $\varphi$ is periodic with order $m$, then
  $$
  (\overline{\Lambda}^{q}_{\varphi,\overline{\gamma}})^m \circ \overline{\rho}(X) \circ (\overline{\Lambda}^{q}_{\varphi,\overline{\gamma}})^{-m}
  = \overline{\rho}\circ \Phi^{q}_{\tau\varphi^m(\tau)}\circ \Psi_{\varphi^m(\tau)\tau}^{q} (X) = \overline{\rho}(X)$$
  for every $X\in \mathcal{T}^{q}_{\tau}$. Then $(\overline{\Lambda}^{q}_{\varphi,\overline{\gamma}})^m$ is a scalar matrix since $\overline{\rho}$ is irreducible.
  Actually we can choose good $\overline{\Lambda}^{q}_{\varphi,\overline{\gamma}}$ such that $(\overline{\Lambda}^{q}_{\varphi,\overline{\gamma}})^m$ is the identity matrix.
  We have all the eigenvalues of $\overline{\Lambda}^{q}_{\varphi,\overline{\gamma}}$ are $m$-roots of unity, and
   $|\text{Trace}\overline{\Lambda}^{q}_{\varphi,\overline{\gamma}}| = 0$ or $|\text{Trace}\overline{\Lambda}^{q}_{\varphi,\overline{\gamma}}| \geq n^{-m}$.

   From lemma 11 in \cite{BW5}, we know the complex dimension of the space of all  periodic edge weight systems for the fixed ideal triangulation is more than or equal to 1. Thus this space is connected. In a local open subset of this space, we can choose $\varphi$-invariant puncture weights such that these puncture weights  smoothly vary according to periodic edge weight systems. Then we have $|\text{Trace}\overline{\Lambda}^{q}_{\varphi,\overline{\gamma}}|$ smoothly varies according to periodic edge weight systems in a local open subset by using the similar argument in complement 10 in \cite{BW4}.
   Since this space is connected and $0$ is an isolated point in the image, we have $|\text{Trace}\overline{\Lambda}^{q}_{\varphi,\overline{\gamma}}| = 0$ for all periodic edge weight systems with the chosen puncture weights, or $|\text{Trace}\overline{\Lambda}^{q}_{\varphi,\overline{\gamma}}| \geq n^{-m}$ for all periodic edge weight systems with the chosen puncture weights.

   If we can find one periodic edge weight system with the chosen puncture weights such that
   $$\lim_{n\text{\;} odd \rightarrow \infty} \frac{1}{n} \log |\text{Trace}\overline{\Lambda}^{q_n}_{\varphi,\overline{\gamma}}| = 0,$$
   we can conclude that the above equation is true for every periodic edge weight system with the chosen puncture weights.

\end{rem}

\bibliography{reference}

\begin{thebibliography}{10}

\bibitem{BX}
Francis Bonahon and Xiaobo Liu.
\newblock {Representations of the quantum Teichm{\"u}ller space and invariants
  of surface diffeomorphisms}.
\newblock {\em Geometry \& Topology}, 11(2):889--937, 2007.

\bibitem{BW1}
Francis Bonahon and Helen Wong.
\newblock {Quantum traces for representations of surface groups in SL2
  ($\mathbb{C}$)}.
\newblock {\em Geometry \& Topology}, 15(3):1569--1615, 2011.

\bibitem{BW2}
Francis Bonahon and Helen Wong.
\newblock {Representations of the Kauffman bracket skein algebra I: invariants
  and miraculous cancellations}.
\newblock {\em Inventiones mathematicae}, 204:195--243, 2016.

\bibitem{BW7}
Francis Bonahon and Helen Wong.
\newblock {The Witten-Reshetikhin-Turaev representation of the Kauffman bracket
  skein algebra}.
\newblock {\em Proceedings of the American Mathematical Society},
  144(6):2711--2724, 2016.

\bibitem{BW3}
Francis Bonahon and Helen Wong.
\newblock {Representations of the Kauffman bracket skein algebra, II: Punctured
  surfaces}.
\newblock {\em Algebraic \& geometric topology}, 17(6):3399--3434, 2017.

\bibitem{BW4}
Francis Bonahon and Helen Wong.
\newblock {Representations of the Kauffman bracket skein algebra III: closed
  surfaces and naturality}.
\newblock {\em Quantum Topology}, 10(2):325--398, 2019.

\bibitem{BW5}
Francis Bonahon, Helen Wong, and Tian Yang.
\newblock {Asymptotics of quantum invariants of surface diffeomorphisms I:
  conjecture and algebraic computations}.
\newblock {\em arXiv preprint arXiv:2112.12852}, 2021.

\bibitem{BW6}
Francis Bonahon, Helen Wong, and Tian Yang.
\newblock {Asymptotics of quantum invariants of surface diffeomorphisms II: The
  figure-eight knot complement}.
\newblock {\em arXiv preprint arXiv:2203.05730}, 2022.

\bibitem{DB}
Doug Bullock.
\newblock {Rings of SL2 ($\mathbb{C}$)-characters and the Kauffman bracket
  skein module}.
\newblock {\em Commentarii Mathematici Helvetici}, 72(4):521--542, 1997.

\bibitem{BJ}
Doug Bullock and Jozef Przytycki.
\newblock {Multiplicative structure of Kauffman bracket skein module
  quantizations}.
\newblock {\em Proceedings of the American Mathematical Society},
  128(3):923--931, 2000.

\bibitem{MC}
Matthew Conder.
\newblock {\em {Discrete and free subgroups of $SL_2$}}.
\newblock PhD thesis, Department of Pure Mathematics and Mathematical
  Statistics, University of Cambridge, 2020.

\bibitem{MP}
Marc Culler and Peter~B Shalen.
\newblock Varieties of group representations and splittings of 3-manifolds.
\newblock {\em Annals of Mathematics}, pages 109--146, 1983.

\bibitem{BM}
Benson Farb and Dan Margalit.
\newblock {\em A primer on mapping class groups (pms-49)}, volume~41.
\newblock Princeton university press, 2011.

\bibitem{FG}
Charles Frohman and R{\u{a}}zvan Gelca.
\newblock Skein modules and the noncommutative torus.
\newblock {\em Transactions of the American Mathematical Society},
  352(10):4877--4888, 2000.

\bibitem{TLe2}
Charles Frohman, Joanna Kania-Bartoszynska, and Thang L{\^e}.
\newblock {Unicity for representations of the Kauffman bracket skein algebra}.
\newblock {\em Inventiones mathematicae}, 215:609--650, 2019.

\bibitem{IDP}
Iordan Ganev, David Jordan, and Pavel Safronov.
\newblock The quantum frobenius for character varieties and multiplicative
  quiver varieties.
\newblock {\em arXiv preprint arXiv:1901.11450}, 2019.

\bibitem{KO}
Oleg~N Karpenkov.
\newblock {Continued Fractions and ${\rm SL}(2,\mathbb Z)$ Conjugacy Classes.
  Elements of Gauss Reduction Theory}.
\newblock In {\em Geometry of Continued Fractions}, pages 115--128. Springer,
  2022.

\bibitem{KK}
Hiroaki Karuo and Julien Korinman.
\newblock Azumaya loci of skein algebras.
\newblock {\em arXiv preprint arXiv:2211.13700}, 2022.

\bibitem{Ka}
Rinat~M Kashaev.
\newblock The hyperbolic volume of knots from the quantum dilogarithm.
\newblock {\em Letters in mathematical physics}, 39(3):269--275, 1997.

\bibitem{TK}
Tsit~Yuen Lam and Ka~Hin Leung.
\newblock On vanishing sums of roots of unity.
\newblock {\em Journal of algebra}, 224(1):91--109, 2000.

\bibitem{Liu}
Xiaobo Liu.
\newblock {The quantum Teichm{\"u}ller space as a noncommutative algebraic
  object}.
\newblock {\em Journal of Knot Theory and its Ramifications}, 18(05):705--726,
  2009.

\bibitem{MM}
Hitoshi Murakami and Jun Murakami.
\newblock The colored jones polynomials and the simplicial volume of a knot.
\newblock 2001.

\bibitem{HTM}
Hitoshi Murakami, Tomotada Ohtsuki, and Masae Okada.
\newblock {Invariants of three-manifolds derived from linking matrices of
  framed links}.
\newblock 1992.

\bibitem{GM}
Gerald Myerson.
\newblock How small can a sum of roots of unity be?
\newblock {\em The American Mathematical Monthly}, 93(6):457--459, 1986.

\bibitem{PS}
J{\'o}zef~H Przytycki and Adam~S Sikora.
\newblock {On Skein Algebras And Sl\_2 (C)-Character Varieties}.
\newblock {\em arXiv preprint q-alg/9705011}, 1997.

\bibitem{NT}
Nurdin Takenov.
\newblock {Representations of the Kauffamn skein algebra of small surfaces}.
\newblock {\em arXiv preprint arXiv:1504.04573}, 2015.

\bibitem{TT}
Terry Tao.
\newblock How small can a sum of a few roots of unity be?

\end{thebibliography}

\end{document}